\documentclass[12pt]{amsart}

\usepackage[usenames,dvipsnames]{color}
\usepackage{anysize}
\marginsize{2cm}{2cm}{2cm}{2cm}
\usepackage{adjustbox}
\usepackage{amsmath, amsthm, amssymb,amsbsy,amsfonts}
\usepackage[latin1]{inputenc}
\usepackage{latexsym}
\usepackage{color}
\usepackage{array}
\usepackage{graphicx}
\newtheorem{thm}{Theorem}[section]
\newtheorem{cor}[thm]{Corollary}
\newtheorem{lem}[thm]{Lemma}
\newtheorem{prop}[thm]{Proposition}
\theoremstyle{definition}
\newtheorem{defn}[thm]{Definition}
\theoremstyle{remark}
\newtheorem{rem}[thm]{Remark}
\numberwithin{equation}{section}

\newcommand{\norm}[1]{\left\Vert#1\right\Vert}

\newcommand{\abs}[1]{\left\vert#1\right\vert}
\newcommand{\set}[1]{\left\{#1\right\}}

\newcommand{\paren}[1]{\left(#1\right)}
\newcommand{\corch}[1]{\left[#1\right]}

\newcommand{\ud}{\mathrm{d}}

\newcommand{\Hil}{\mathcal{H}}
\newcommand{\Ka}{\mathcal{K}}
\newcommand{\Vsb}{\mathcal{V}}
\newcommand{\Wsb}{\mathcal{W}}

\newcommand{\R}{\mathbb{R}}
\newcommand{\Z}{\mathbb{Z}}
\newcommand{\N}{\mathbb{N}}

\newcommand{\e}{\epsilon}
\newcommand{\dl}{\delta}

\newcommand{\ga}{\gamma}
\newcommand{\Leb}{L^{1}\left(\mathbb{R}\right)}
\newcommand{\Lebcuad}{L^{2}\left(\mathbb{R}\right)}

\newcommand{\lcuad}{\ell^2(\N)}
\newcommand{\lcuadent}{\ell^2(\Z)}

\allowdisplaybreaks

\begin{document}

\title[Approximate oblique dual frames]{Approximate oblique dual frames}

\author[Jorge P. D\'{i}az]{Jorge P. D\'{i}az$^{1}$}%

\author[Sigrid B. Heineken]{Sigrid B. Heineken$^{2}$}%
%%{Sigrid B. Heineken\\Departamento de Matem\'atica \\Facultad de Ciencias Exactas y Naturales\\
%%Universidad de Buenos Aires\\ Pabell\'on I\\ Pabell\'on I\\
%%C1428EGA C.A.B.A.\\ Argentina\\ IMAS, UBA-CONICET\\Argentina.}
%%\email{sheinek@dm.uba.ar}%

\author[Patricia M. Morillas]{Patricia M. Morillas$^{1,*}$\\\\ $^{1}$\textit{I\lowercase{nstituto de }M\lowercase{atem\'{a}tica }A\lowercase{plicada} S\lowercase{an} L\lowercase{uis, }UNSL-CONICET, E\lowercase{j\'{e}rcito de los }A\lowercase{ndes 950, 5700 }S\lowercase{an} L\lowercase{uis,} A\lowercase{rgentina} \\ $^2$ D\lowercase{epartamento de} M\lowercase{atem\'atica}, FCE\lowercase{y}N, U\lowercase{niversidad de }B\lowercase{uenos} A\lowercase{ires}, P\lowercase{abell\'on} I, C\lowercase{iudad }U\lowercase{niversitaria}, IMAS, UBA-CONICET, C1428EGA C.A.B.A., A\lowercase{rgentina}}}
%%\address{Patricia M. Morillas\\Instituto de Matem\'{a}tica Aplicada San Luis\\
%%         UNSL-CONICET\\
%%         and Departamento de Matem\'{a}tica\\
%%         FCFMyN, UNSL\\
%%         Ej\'{e}rcito de los Andes 950, 5700 San Luis, Argentina}%
%%\email{morillas@unsl.edu.ar}%
%%
%
%%\date{}%
%%\dedicatory{}%
%%\commby{}%
%
\thanks{* Corresponding author. E-mail address: morillas.unsl@gmail.com\\
\textit{E-mail addresses:} jpdiaz1179@gmail.com (J. P. D\'{i}az),
sheinek@dm.uba.ar (S. B. Heineken), morillas.unsl@gmail.com (P. M.
Morillas).}
%% ---------------------------------------------------------------
%
\begin{abstract}
In representations using frames, oblique duality appears in
situations where the analysis and the synthesis has to be done in
different subspaces. In some cases, we cannot obtain an explicit
expression for the oblique duals and in others there exists only one
oblique dual frame which has not the properties we need. Also, in
practice the computations are not exact. To give a solution to these
problems, in this work we introduce and investigate the notion of
approximate oblique dual frames first in the setting of separable
Hilbert spaces. We present several properties and provide different
characterizations of approximate oblique dual frames. We focus then
on approximate oblique dual frames in shift-invariant subspaces of
$\Lebcuad$ and give different conditions on the generators that
assure their existence. The importance of approximate oblique dual
frames from a numerical and computational point of view is
illustrated with an example of frame sequences generated by
$B$-splines, where the previous results are used to construct
approximate oblique dual frames which have better attributes than
the exact ones. We provide an expression for the approximation error
and study its behaviour.

\bigskip

\bigskip


\medskip


\end{abstract}

\maketitle

% ----------------------------------------------------------------
\section{Introduction}

A frame is a sequence of vectors in a separable Hilbert space.
Frames generalize bases \cite{Christensen (2016), Daubechies (1992),
Daubechies-Grossmann-Meyer (1986), Duffin-Schaeffer (1952), Heil
(2006), Young (2001)}. The main difference to bases is that the
reconstruction of each vector is not necessarily unique, i.e. there
exists more than one collection of coefficients to represent each
vector as a combination of the elements of a frame. These frame
coefficients are associated to other sequences called dual frames.

In practice, the computation of duals is not exact and sometimes it
is even not possible to give in theory an analytic expression of it.
Another difficulty is that in applications we need to truncate the
frame representations in order to work with finite sequences of
numbers. Also the frame coefficients can in general only be computed
approximately. Approximate dual frames appear as an answer to these
problems \cite{Christensen-Laugesen (2010), Dorfler-Matusiak (2015),
Feichtinger-Onchis-Wiesmeyr (2014),
Perraudin-Holighaus-Sondergaard-Balazs (2018)}. They are also easier
to build and can be adapted to our needs.

For frames in a subspace, the reconstruction can also be done with
coefficients that depend on duals that do not necessarily belong to
the same subspace. These are known as oblique dual frames
\cite{Eldar-Christensen (2006), Christensen-Eldar (2004), Eldar
(2003a), Eldar (2003b), Eldar-Werther (2005), Kim-Koo-Lim (2013),
Koo-Lim (2013), Koo-Lim (2015), Xiao-Zhu-Zeng (2013)}. These dual
frames are related with oblique projections and are of particular
interest in signal processing, where the goal is to recover the
signal itself. In this situation we again have the limitations and
the computational difficulties described before. Moreover, sometimes
the subspaces are given by the problem and there exists a unique
oblique dual frame. This unique oblique dual could need to be
improved in some aspects. In order to respond to these different
problems, in this work we introduce the concept of approximate
oblique dual frames and study its properties, first in the setting
of separable Hilbert spaces and then in shift-invariant subspaces of
$L^{2}\paren{\mathbb{R}}$. Since in general there exists more than
one approximate oblique dual, we also gain freedom in its
construction.

In Section 2 we give a brief review of existing definitions and
results. In Sections 3 to 6, we introduce and investigate the notion
of approximate oblique dual frames in the setting of separable
Hilbert spaces. In Sections 7 and 8, we focus on approximate oblique
dual frames in shift-invariant subspaces of $\Lebcuad$.

In Section 3, we introduce the notion of $\epsilon$-approximate
coblique dual frames, where $\epsilon>0$, and investigate some of
its properties. We also give its interpretation in the setting of
signal processing theory and the conditions of uniqueness and
consistency in order to have a good approximation of a signal.

In Section 4, we obtain characterizations of approximate oblique
dual frames, in terms  of series expansions as well as in terms of
operators.

In Section 5, we extend to the oblique setting an important property
of approximate dual frames, that is to obtain a reconstruction of
the vectors as close as we desire starting from any pair of
approximate oblique dual frames. We also show how to construct
oblique dual frames from approximate oblique dual frames.

In Section 6, we obtain approximate oblique dual frames from the
perturbation of oblique duals. We prove that if two frames are
``close'', each oblique dual frame of any of them is an approximate
oblique dual frame of the other, expressing its bounds in terms of
the original ones.

Shift-invariant subspaces with frames of translates are widely used
in the applications, and in particular in image and signal
processing. Moreover, shift-invariant spaces play a key role in the
construction of wavelets frames and Gabor frames. In Section 7, we
consider shift-invariant subspaces of $\Lebcuad$ and obtain
conditions on their generators to provide an approximate
reconstruction in one of them. The technique introduced in the proof
of the main result of this section
(Theorem~\ref{caractdualestrasla}) is the key to the obtainment of
the central results of section 8 about approximate oblique dual
frames.

We begin Section 8 stating sufficient conditions for approximate
oblique duality, that are derived from results of section 7. We then
give an expression for the Fourier transform of the oblique
projection when the spaces are shift-invariant. We use it to provide
a more manageable sufficient condition, and also a necessary
condition on the generators of shift-invariant subspaces for the
existence of approximate oblique dual frames.

In Section 9 we  describe the necessity of working in the
applications with approximate oblique dual frames. The importance of
approximate oblique dual frames from a numerical and computational
point of view is illustrated with an example. We consider frame
sequences of translates generated by $B$-splines. In this case, the
generator of the smooth unique oblique dual frame has not compact
support whereas the generator of each obtained approximate oblique
dual frame is not only smooth but also compactly supported. We give
an expression for the approximation error and analyze its behaviour.

% ----------------------------------------------------------------

\section{Preliminaries}

We consider $\Hil, \Ka$ separable Hilbert spaces. The space of
bounded operators from $\Hil$ to $\Ka$ will be denoted by
$L(\Hil,\Ka)$. For $T\in L(\Hil,\Ka)$ denote the image, the null
space and the adjoint of $T$ by  $\mathcal{R}(T)$, $\mathcal{N}(T)$
and $T^{*}$, respectively. If $T$ has closed range we also consider
the Moore-Penrose pseudo-inverse of $T$ denoted by $T^{\dagger}$.
The inner product and the norm in $\Hil$ will be denoted by $\langle
\cdot ,\cdot \rangle$ and $\|\cdot\|$, respectively.

The Lebesgue measure of a measurable subset $A$ of $\mathbb{R}$ will
be denoted by $|A|$. We denote the characteristic function of $A$
with $\chi_{A}$ and the complement of $A$ with $A^{c}$. We consider
the class $\mathcal{C}_{{\rm per}}^{p}$ of $1$-periodic functions
that restricted to $[0,1)$ belong to $L^{p}(0,1)$.

We now review briefly definitions and present properties that we use
later. In the sequel $\mathcal{V}$ and $\mathcal{W}$ will be two
closed subspaces of $\Hil$.

\subsection{Oblique projections}

Let $\Hil=\mathcal{W} \oplus \mathcal{V}^{\perp}$. The oblique
projection onto $\mathcal{W}$ along $\mathcal{V}^{\perp}$, is the
unique operator that satisfies

\centerline{$\pi_{\mathcal{W}\mathcal{V}^{\perp}}f=f \text{ for all
} f\in \mathcal{W}$,\,\, \, $\pi_{\mathcal{W}\mathcal{V}^{\perp}}f=0
\text{ for all } f\in \mathcal{V}^{\perp}$.}

Equivalently,
$\mathcal{R}(\pi_{\mathcal{W}\mathcal{V}^{\perp}})=\mathcal{W}$ and
$\mathcal{N}(\pi_{\mathcal{W}\mathcal{V}^{\perp}})=\mathcal{V}^{\perp}$.
Observe that $\left(\pi_{\Vsb\Wsb^{\perp}}\right)^{\ast} =
\pi_{\Wsb\Vsb^{\perp}}$ and
$(\pi_{\Vsb\Wsb^{\perp}})_{\mid\Vsb}=I_{\Vsb}$. If
$\mathcal{V}=\mathcal{W}$ we obtain the orthogonal projection onto
$\mathcal{V}$, which we denote by $P_{\Vsb}$. The next result can be
deduced from \cite[Lemma 2.1]{Heineken-Morillas (2018)} and the
definitions of orthogonal and oblique projections:
\begin{lem}\label{lemma 001}
Let $\Hil=\mathcal{W} \oplus \mathcal{V}^{\perp}$. The following
holds:
\begin{itemize}
\item[(i)]
$ (P_{\Wsb})_{ \mid\Vsb } :\Vsb \rightarrow\Wsb$ and $ (P_{\Vsb }
)_{ \mid\Wsb } :\Wsb \rightarrow\Vsb $ are isomorphisms, $ ((
P_{\Wsb})_{ \mid\Vsb })^{-1} = (\pi_{\Vsb\Wsb^{\perp}})_{ \mid \Wsb
}$ and $ ((P_{\Vsb })_{ \mid\Wsb })^{-1} = (\pi_{
\Wsb\Vsb^{\perp}})_{ \mid\Vsb }$.
\item[(ii)] $ ((P_{\Wsb})_{ \mid\Vsb })^{\ast} = (P_{\Vsb })_{ \mid
\Wsb }$ and $((\pi_{\Wsb\Vsb^{\perp}})_{ \mid\Vsb })^{\ast} = (
\pi_{\Vsb\Wsb^{\perp}})_{ \mid\Wsb }$.\qedhere
\end{itemize}
\end{lem}
The concept of angle between two closed subspaces that we will use
is the following:
\begin{defn}
Given $\Wsb$ and $\Vsb$ two closed subspaces of $\Hil $, we define
the \emph{angle from $\Vsb$ to $\Wsb$} as the unique real number
$\theta(\Vsb,\Wsb)\in  [0,\frac{\pi}{2}]$ such that
\begin{equation}\label{angentsub}
\cos\theta(\Vsb,\Wsb) = \inf_{v\in\Vsb,|| v || = 1}||P_{\Wsb}v||.
\qedhere
\end{equation}
\end{defn}
\begin{thm}\cite{Tang (2000)}\label{descompdH}
Let $\Vsb$, $\Wsb$ be closed subspaces of a separable Hilbert space
$\Hil $. Then the following assertions are equivalent:
\begin{itemize}
\item[(i)] $\Hil=\Wsb\oplus\Vsb^{\perp}.$
\item[(ii)] $\Hil =\Vsb \oplus\Wsb^{\perp}.$
\item[(iii)] $\cos\theta(\Vsb,\Wsb)>0 $ and $\cos\theta(\Wsb,\Vsb)>0.$
\end{itemize}
\end{thm}

We have $\norm{\pi_{\Wsb\Vsb^{\perp}} } =
\frac{1}{\cos\theta(\Wsb,\Vsb)}\geq1$ \cite[Lemma 2.4]{BGM (2019)}.

\subsection{Frames}

Frames have been introduced by Duffin  and Schaeffer in
\cite{Duffin-Schaeffer (1952)}. Using a frame, each element of a
Hilbert space has a representation which in general is not unique.
This flexibility makes them attractive for many applications
involving signal expansions.

We will now recall the definition of frame for a closed subspace of
$\Hil$.

\begin{defn}\label{D frame}
Let $\mathcal{W}$ be a closed subspace of $\Hil$ and
$\{f_k\}_{k=1}^{\infty} \subset \mathcal{W}$. Then
$\{f_k\}_{k=1}^{\infty}$ is a \emph{frame} for $\mathcal{W}$, if
there exist constants $0 < \alpha \leq \beta < \infty$ such that
\begin{equation}\label{E cond frame}
\alpha\|f\|^{2} \leq \sum_{k=1}^{\infty}|\langle f,f_{k}\rangle
|^{2} \leq \beta\|f\|^{2}  \text{ for all $f\in \mathcal{W}$.}
\end{equation}
\end{defn}

If the right inequality in (\ref{E cond frame}) is satisfied,
$\{f_k\}_{k=1}^{\infty}$ is a \emph{Bessel sequence} for
$\mathcal{W}$. The constants $\alpha$ and $\beta$ are the \emph{
frame bounds}. In case $\alpha=\beta,$ we call
$\{f_k\}_{k=1}^{\infty}$ an \emph{$\alpha$-tight frame}, and if
$\alpha=\beta=1$ it is a \emph{Parseval frame} for $\mathcal{W}$.

To a Bessel sequence $\mathcal{F}=\{f_k\}_{k=1}^{\infty}$ for
$\mathcal{W}$ we associate the \emph{synthesis operator}

\centerline{$T_{\mathcal{F}}:\ell^2\rightarrow \Hil,$
$T_{\mathcal{F}}\{c_{k}\}_{k=1}^{\infty}=\sum_{k=1}^{\infty}c_{k}f_{k},$}

\noindent the \emph{analysis operator}

\centerline{$T_{\mathcal{F}}^{*}: \Hil\rightarrow \ell^2$,
$T_{\mathcal{F}}^{*}f=\{\langle f,f_{k}\rangle\}_{k=1}^{\infty},$}

\noindent and the \emph{frame operator}

\centerline{$S_{\mathcal{F}}=T_{\mathcal{F}}T_{\mathcal{F}}^{*}$.}

\noindent A Bessel sequence $\mathcal{F}=\{f_k\}_{k=1}^{\infty}$ for
$\mathcal{W}$ is a frame for $\mathcal{W}$ if and only
$\mathcal{R}(T_{\mathcal{F}})=\mathcal{W}$, or equivalently,
$S_{\mathcal{F}}$ is invertible when restricted to $\mathcal{W}$.
Furthermore, $\mathcal{F}$ is an $\alpha$-tight frame for
$\mathcal{W}$ if and only if $S_{\mathcal{F}}=\alpha P_{\Wsb}$. A
\emph{Riesz basis} for $\mathcal{W}$ is a frame for $\mathcal{W}$
which is also a basis. If $\{f_k\}_{k=1}^{\infty}$ is a frame for
$\overline{{\rm span}}\{f_k\}_{k=1}^{\infty}$ we say that
$\{f_k\}_{k=1}^{\infty}$ is a \emph{frame sequence}.

The \emph{optimal upper frame bound}, i.e., the infimum over all
upper frame bounds is $||S_{\mathcal{F}}||=||T_{\mathcal{F}}||^{2}$,
and the \emph{optimal lower frame bound}, i.e., the supremum over
all lower frame bounds is
$||S_{\mathcal{F}}^{\dag}||^{-1}=||T_{\mathcal{F}}^{\dagger}||^{-2}$.
\begin{defn}
Let $\mathcal{W}$ be a closed subspace of $\Hil$. Let
$\mathcal{F}=\{f_k\}_{k=1}^{\infty}$ and
$\mathcal{G}=\{g_{k}\}_{k=1}^{\infty}$ be frames for $\mathcal{W}$.
If $T_{\mathcal{G}}T_{\mathcal{F}}^{*}=P_{\mathcal{W}},$ we say that
$\mathcal{G}$ is a {\em dual frame} of $\mathcal{F}$ in
$\mathcal{W}$.
\end{defn}
The sequence $\{S_{\mathcal{F}}^{\dagger}f_k\}_{k=1}^{\infty}$ is
the \emph{canonical dual frame} of $\{f_k\}_{k=1}^{\infty}$ in
$\mathcal{W}$.

We recall the definition of \emph{oblique dual frames} \cite{Eldar
(2003a), Christensen-Eldar (2004)}:
\begin{defn}
Let $\Hil=\mathcal{W} \oplus \mathcal{V}^{\perp}$. Let $\mathcal{F}
= \{f_k\}_{k=1}^{\infty}$ and $\mathcal{G} = \{
g_k\}_{k=1}^{\infty}$ be frames for $\Wsb $ and $\Vsb $,
respectively. We call $\mathcal{F}$ and $\mathcal{G}$ \emph{oblique
dual frames} if $T_{\mathcal{F}}T_{\mathcal{G}}^{*}
=\pi_{\Wsb\Vsb^{\perp} }$  or $T_{\mathcal{G}} T_{
\mathcal{F}}^{*}=\pi_{\Vsb\Wsb^{\perp} }.$
\end{defn}
In this case we also say that $\{g_k\}_{k=1}^{\infty}$ is an
\emph{oblique dual frame} of $\{f_k\}_{k=1}^{\infty}$ in $\Vsb $ and
that $\{f_k\}_{k=1}^{\infty}$ is an \emph{oblique dual frame} of
$\{g_k \}_{k=1}^{\infty}$ in $\Wsb $. The \emph{canonical oblique
dual frame} of $\{f_k\}_{k=1}^{\infty}$ in $\Vsb $ is
$\mathcal{G}=\{\pi_{\Vsb\Wsb^{\perp}}S_{\mathcal{F}}^{\dag}f_k\}_{k=1}^{\infty}$
and
$T_{\mathcal{G}}=\pi_{\Vsb\Wsb^{\perp}}S_{\mathcal{F}}^{\dag}T_{\mathcal{F}}$
\cite[Theorem 3.2.]{Christensen-Eldar (2004)}.

Let $\Hil=\Wsb\oplus\Vsb^{\perp}$. Let $\mathcal{F} =
\{f_k\}_{k=1}^{\infty}$ and $\mathcal{G} = \{ g_k\}_{k=1}^{\infty}$
be Bessel sequences for $\Wsb $ and $\Vsb $, respectively. We will
use the next equalities that follow directly from the definitions of
the operators:

\centerline{$T_{
\mathcal{F}}^{*}=T_{\mathcal{F}}^{*}\pi_{\Vsb\Wsb^{\perp}}$}

\noindent and

\begin{equation}\label{E TFTGpi}
(T_{\mathcal{F}}T_{\mathcal{G}}^{*})_{\mid\Vsb}=(T_{\mathcal{F}}T_{\mathcal{G}}^{*})_{\mid\Wsb}(\pi_{\Wsb\Vsb^{\perp}})_{\mid\Vsb}.
\end{equation}

For more details about frames we refer the reader to
\cite{Christensen (2016), Daubechies (1992), Young (2001)}.

\subsection{Shift-invariant frame sequences}

For  $f\in\Leb$, its \emph{Fourier transform}, that we denote by
$\widehat{f}$, is defined by

\centerline{$\widehat{f}(\ga) = \int_{-\infty}^{\infty} f(x)
e^{-i2\pi x\ga} \ud x$}

\noindent for $\ga\in\R $. The  Fourier transform can be extended as
a unitary operator in $\Lebcuad$. Given $ k\in \Z $, the
\emph{translation operator} is

\centerline{$ T_k : \Lebcuad \rightarrow \Lebcuad$, $\quad T_k f(x)
= f(x-k),$}

\noindent and $\widehat{T_{k}f}(\ga)= e^{-i2\pi k
\ga}\widehat{f}(\ga)$ holds.

Let $\phi\in\Lebcuad$ be such that $\{T_k\phi\}_{k\in\Z}$ is a
Bessel sequence. Let $f\in\Lebcuad$. Then $f=\sum_{k\in\Z}c_k
T_k\phi$ with $\{c_k\}_{k\in\Z}\in \lcuadent$ if and only if its
Fourier transform is $\widehat{f} = H \widehat{\phi}$, where $H \in
\mathcal{C}_{{\rm per}}^{2}$ and restricted to $[0,1)$ is equal to
$\sum_{k\in\Z} c_k e^{-i2\pi k\cdot}$ \cite[Lemma 9.2.2]{Christensen
(2016)}.

Let $\mathcal{W}:=\overline{\textrm{span}}\{T_k\phi\}_{k\in\Z}$. A
space of this type is called \emph{shift-invariant}. If
$\{T_k\phi\}_{k\in\Z}$ is a frame sequence, then

\centerline{$\mathcal{W}=\{\sum_{k\in\Z}c_k T_k\phi:\{c_k\}_{k\in\Z}
\in \lcuadent\}$.}

For $\phi, \phi_{1}\in\Lebcuad$, we use the bracket notation

\centerline{$[\widehat{\phi},\widehat{\phi}_{1}]: \mathbb{R}
\rightarrow \mathbb{R}, \quad
[\widehat{\phi},\widehat{\phi}_{1}](\gamma) = \sum_{k\in
\mathbf{Z}}\widehat{\phi}(\gamma+k)\overline{\widehat{\phi}_{1}(\gamma
+ k)}$,}

\noindent we denote

\centerline{$\Phi : \mathbb{R} \rightarrow \mathbb{R}, \quad\Phi(
\gamma) = \sum_{k\in\Z}|\widehat{\phi}(\gamma + k)|^2$}

\noindent and $N(\Phi)=\set{\gamma:\Phi(\gamma) = 0}$. The functions
$[\widehat{\phi},\widehat{\phi}_{1}]$ and $\Phi$ belong to
$\mathcal{C}_{{\rm per}}^{1}$.

\begin{thm}\label{T Benedetto}
Let $\phi\in\Lebcuad$. Then
\begin{itemize}
\item[(i)]
$\{T_{k} \phi\}_{k\in \mathbb{Z}}$ is a Bessel sequence with bound
$\beta $ if and only if $\Phi(\ga) \leq \beta$ a.e. $\ga \in
\mathbb{R}$.
\item[(ii)]
$\{T_{k} \phi\}_{k\in \mathbb{Z}}$ is a frame sequence with bounds
$\alpha, \beta $ if and only if $\alpha \leq\Phi(\ga) \leq \beta $
a.e. $\ga\in N(\Phi)^{c}$. \qedhere
\end{itemize}
\end{thm}

Let $\phi, \phi_{1}\in\Lebcuad$ be such that $\{ T_{k} \phi\}_{k\in
\mathbb{Z}}$ and $\{T_{k}\phi_{1} \}_{k \in \mathbb{Z}}$ are Bessel
sequences with bounds $\beta $ and $\widetilde{\beta}$,
respectively. By Theorem~\ref{T Benedetto}(i),
\begin{equation}\label{E cota superior BB}
[\,|\widehat{\phi}|,|\widehat{\phi}_{1}|\,](\ga) \leq
\sqrt{\beta\widetilde{\beta}}\,\, \text{ a.e. } \gamma\in\mathbb{R}.
\end{equation}
\begin{prop}\label{cond.L2.en.suma.dir}\cite{Christensen-Eldar
(2004)} Let $\phi, \phi_{1}\in\Lebcuad$, and assume that
$\{T_k\phi\}_{k\in \mathbb{Z}}$ and $\{T_k\phi_{1} \}_{k \in
\mathbb{Z}}$ are frame sequences. Let
$\mathcal{W}:=\overline{\textrm{span}}\{T_k\phi\}_{k\in\Z}$ and
$\mathcal{V}:=\overline{\textrm{span}}\{T_k\phi_{1}\}_{k \in \Z}$.
Then the following are equivalent:
\begin{itemize}
\item[(i)]
$\Lebcuad =\Wsb \oplus\Vsb^{\perp}$.
\item[(ii)]
$N(\Phi) = N(\Phi_{1})$ and there exists a constant $c>0$ such that

\centerline{$|\,[\widehat{\phi},\widehat{\phi}_{1}](\ga)\,|\geq c$
a.e. $\ga\in N(\Phi)^{c}$.}
\end{itemize}
\end{prop}
Let $\Wsb$ and $\Vsb$ be two closed subspaces of $\Lebcuad$ such
that $\Lebcuad =\Wsb \oplus\Vsb^{\perp}$. In \cite[Proposition
4.8]{Christensen-Eldar (2004)} it is shown that if
$\{T_k\phi\}_{k\in\Z}$ is a frame for $\Wsb$ and
$\{T_k\phi_{1}\}_{k\in\Z}$ is a frame for $\Vsb$ then
\begin{equation}\label{E cos marco}
\cos\theta(\Wsb,\Vsb)\geq ess \inf_{\gamma\in N(\Phi)^{c}}
\frac{|\,[\widehat{\phi},\widehat{\phi}_{1}](\ga)\,|}{
\sqrt{\Phi(\gamma)\Phi_{1}(\gamma)}}.
\end{equation}

%%%%%%%%%%%%%%%%%%%%%%%%%%%%%%%%%%%%%%%%%%%%%%%
\section{Definition and fundamental properties}

In this section we present the concept of approximate oblique
duality and explore some of its properties. First we introduce the
following definition of approximate oblique dual frames:
\begin{defn}\label{D marcos duales oblicuos aproximados}
Let $\Hil=\mathcal{W} \oplus \mathcal{V}^{\perp}$. Let $\mathcal{F}
= \{f_k\}_{k=1}^{\infty}$ and $\mathcal{G}= \{g_k\}_{k=1}^{\infty}$
be frames for $\Wsb$ and $\Vsb$, respectively. Let $\e\geq0$. We say
that $\mathcal{F}$ and $\mathcal{G}$ are \emph{$\e$-approximate
oblique dual frames} if
\begin{equation}\label{e001}
\norm{\pi_{\Wsb\Vsb^{\perp} } - T_{\mathcal{F}}T_{\mathcal{G}}^{
*}} \leq \e \textrm{  or  } \norm{\pi_{\Vsb\Wsb^{\perp}} - T_{\mathcal{G}} T_{\mathcal{F}}^{ *}} \leq \e. \qedhere
\end{equation}
\end{defn}
In this case we also say that $\{g_k\}_{k=1}^{\infty}$ is an
\emph{$\e$-approximate oblique dual frame} of $\{f_k
\}_{k=1}^{\infty}$ in $\Vsb$ and that $\{f_k\}_{k=1}^{\infty}$ is an
\emph{$\e$-approximate oblique dual frame} of $\{g_k
\}_{k=1}^{\infty}$ in $\Wsb$. If $\Vsb =\Wsb$, then $\mathcal{F}$
and $\mathcal{G}$ will be called $\e$-approximate dual frames in
$\Wsb$.

Note that the conditions of the definition can be written as
\begin{equation}\label{E D 1}
\norm{\pi_{\Wsb\Vsb^{\perp}}f - \sum_{ k=1 }^{\infty} \langle f, g_k
\rangle f_k }\leq \e\norm{f } \text{ for all } f\in\Hil
\end{equation}
and
\begin{equation}\label{E D 2}
\norm{\pi_{\Vsb\Wsb^{\perp}}f - \sum_{ k=1 }^{\infty} \langle f, f_k
\rangle g_k } \leq \e\norm{f } \text{ for all } f\in\Hil,
\end{equation}
respectively.
\begin{rem}\label{Obs definicion de mdoa}
\begin{enumerate}
\item[(i)]
If $\e = 0 $, $\mathcal{F}$ and $\mathcal{G}$ are \emph{oblique dual
frames}. If $\Vsb =\Wsb = \Hil $ and $\e < 1 $, then $\mathcal{F}$
and $\mathcal{G}$ are \emph{approximate dual frames} as defined in
\cite{Christensen-Laugesen (2010)}.

\item[(ii)]\label{vwh}
If $\e < 1 $, by Neumann's Theorem
$(T_{\mathcal{F}}T_{\mathcal{G}}^{*}) _{|\Wsb}$ is an invertible
operator from $\Wsb$ to $\Wsb$. So, by (\ref{E TFTGpi}) and
Lemma~\ref{lemma 001}(i),
$(T_{\mathcal{F}}T_{\mathcal{G}}^{*})_{|\Vsb}$ is an invertible
operator from $\Vsb$ to $\Wsb$.

\item[(iii)]
If $\e < 1 $, it is sufficient that $\{f_k\}_{k=1}^{\infty}$ and
$\{g_k\}_{k=1}^{\infty}$ are Bessel sequences, since $\mathcal{R}(
T_{\mathcal{F}}) =\Wsb$ by (ii). Hence $\{f_k\}_{k=1}^{\infty}$ is a
frame for $\Wsb$. Analogously $\{g_k\}_{k=1}^{\infty}$ is a frame
for $\Vsb$.

\item[(iv)] Assume that $\e<1$. If $f\in\Wsb$, by (ii),
$$f=({T_{\mathcal{F}}(T_{\mathcal{G}}^{*})} _{|\Wsb})^{-1} T_{\mathcal{F}}T_{\mathcal{G}}^{*} f =
\sum \langle f, g_k \rangle ({T_{\mathcal{F}}(T_{\mathcal{G}}^{*})}
_{|\Wsb})^{-1} f_k .$$ Similarly, for $f\in\Vsb$ we have
$$f=({T_{\mathcal{G}} (T_{\mathcal{F}}^{*})}_{ |\Vsb })^{ -1
} T_{\mathcal{G}} T_{\mathcal{F}}^{*} f = \sum \langle f, f_k
\rangle ({T_{\mathcal{G}} (T_{\mathcal{F}}^{*})}_{ |\Vsb })^{ -1 }
g_k. $$ Hence $\{({T_{\mathcal{F}}(T_{\mathcal{G}}^{*})}
_{|\Wsb})^{-1}
 f_k\}_{k=1}^{\infty}$ and $\{g_k\}_{k=1}^{\infty}$ ($\{f_k\}_{k=1}^{\infty}$
and $\{({T_{\mathcal{G}} (T_{\mathcal{F}}^{*})}_{ |\Vsb })^{ -1 }
g_k\}_{k=1}^{\infty}$) are oblique dual frames by \cite[Lemma
3.1(i)]{Christensen-Eldar (2004)}.

\end{enumerate}

\end{rem}

The next proposition tells us that approximate oblique duality is
preserved under the action of unitary operators.
\begin{prop}
Let $\Hil=\mathcal{W} \oplus \mathcal{V}^{\perp}$. Let $\mathcal{F}
= \{f_k\}_{k=1}^{\infty} \subset\Wsb$ and $\mathcal{G} =
\{g_k\}_{k=1}^{\infty} \subset\Vsb $ be $\e$-approximate oblique
dual frames with $\e \geq0 $. Let $ U : \Hil \rightarrow \Hil $ a
unitary operator. Then $ U \mathcal{F} = \{U f_k \}_{k=1}^{\infty}$
and $ U \mathcal{G} = \{U g_k\}_{k=1}^{\infty}$ are $\e$-approximate
oblique dual frames. \qedhere
\end{prop}
\begin{proof}
By \cite[Corollary 5.3.4]{Christensen (2016)}, the sequences $ U
\mathcal{F}$ and $ U \mathcal{G}$ are frames for $ U\Wsb$ and $ U
\Vsb $, respectively. We have $ T_{ U \mathcal{ F}} = U T_{
\mathcal{ F}}$ and $ T_{ U \mathcal{ G}} = U T_{\mathcal{ G}}$.
Since $\Hil=\Wsb\oplus\Vsb^{\perp}$ and $ I = U U^{*}$, $I = U
\pi_{\Wsb\Vsb^{\perp}} U^{*} + U \pi_{\Vsb^{\perp}\Wsb }U^{*}$ and
$\Hil = U\Wsb \oplus (U\Vsb)^{\perp}$. From here, we have $\pi_{
U\Wsb (U\Vsb)^{\perp}} = U \pi_{\Wsb\Vsb^{\perp}} U^{\ast}$. Then
\begin{equation*}
\norm{\pi_{ U\Wsb (U\Vsb)^{\perp}} - T_{ U \mathcal{ F}} T_{ U
\mathcal{ G}}^{\ast}}= \norm{U (\pi_{\Wsb\Vsb^{\perp}} - T_{
\mathcal{ F}} T_{\mathcal{ G}}^{\ast}) U^{ *}} = \norm{\pi_{\Wsb
\Vsb^{\perp}} - T_{ \mathcal{ F}} T_{\mathcal{ G}}^{\ast}} \leq \e,
\end{equation*}
and so $ U \mathcal{F}$ and $ U \mathcal{G}$ are $\e$-approximate
oblique dual frames. \qedhere
\end{proof}
Now we briefly discuss the interpretation of the approximate oblique
duality in the context of signal processing theory. Let
$\mathcal{F}$ be a frame for $\mathcal{W}$. Assume that the samples
$T_{\mathcal{F}}^{*}f=\{\langle f,f_{k}\rangle)\}_{k=1}^{\infty}$ of
an unknown signal $f\in \mathcal{H}$ are given. Our goal is the
reconstruction of $f$ from these samples using a frame $\mathcal{G}$
for $\mathcal{V}$ in such a way that the reconstruction $f_{r}\in
\mathcal{V}$ is a good approximation of $f$. This means on one hand
that there don't exist two different signals in $\Vsb$ with the same
samples. On the other, that the samples of the reconstruction
$f_{r}$ and those of the original signal $f$ are close. Specifically
the following two conditions are required. Let $\e  \geq 0 $,
\begin{enumerate}
\item[(i)]
\textit{Uniqueness of the reconstruction}: If $ f, g\in\Vsb$ and $
T_{ \mathcal{ F}}^{*} f = T_{ \mathcal{ F}}^{*} g $, then $f=g $.

\item[(ii)]
\textit{$\e $-consistent reconstruction}: For every $f\in \Hil $,
$\norm{T_{ \mathcal{ F}}^{*} f_{r} - T_{ \mathcal{ F}}^{*} f } \leq
\e \norm{f }$.
\end{enumerate}
If (ii) holds, we say that $ f_{r}$ is an \emph{ $\e$-consistent
reconstruction} of $ f $ in $\Vsb$. Note that a $ 0 $-consistent
reconstruction is a \emph{consistent reconstruction} as defined in
\mbox{\cite{Eldar (2003a), Eldar (2003b), Eldar-Werther (2005)}}.
The condition (i) is equivalent to $\Vsb \cap\Wsb^{\perp} = \{0 \}$.

If $\Hil=\mathcal{W} \oplus \mathcal{V}^{\perp}$, using an oblique
dual frame of $\mathcal{ F }$ in $\Vsb$, we can obtain the following
bound for $\norm{f_{r} - \pi_{\Vsb\Wsb^{\perp}}f }$:

\begin{prop}\label{P cota reconstr consist dual G}
Let $\Hil=\mathcal{W} \oplus \mathcal{V}^{\perp}$. Let $\mathcal{F}
= \{f_k\}_{k=1}^{\infty}$ be a frame for $\Wsb$ and $\mathcal{G}$ be
an oblique dual frame of $\mathcal{ F }$ in $\Vsb$. Let $\e\geq0$.
If $f\in \Hil$ and $ f_{r}$ is an $\e $-consistent reconstruction of
$f$ in $\Vsb$, then

\centerline{$\norm{f_{r} - \pi_{\Vsb\Wsb^{\perp}}f } \leq \e \norm{
T_{\mathcal{ G}} }\norm{f }$.}\qedhere
\end{prop}

\begin{proof}
Using that $T_{\mathcal{ G}}T_{ \mathcal{
F}}^{*}=\pi_{\Vsb\Wsb^{\perp}}$ and $f_{r}\in\Vsb$, we have

\centerline{$\norm{f_{r} - \pi_{\Vsb\Wsb^{\perp}}f } = \norm{
T_{\mathcal{ G}} T_{ \mathcal{ F}}^{\ast} f_{r} - T_{\mathcal{ G}}
T_{ \mathcal{ F}}^{\ast} f }\leq \norm{T_{\mathcal{ G}} } \norm{T_{
\mathcal{ F}}^{\ast} f_{r} - T_{ \mathcal{ F}}^{\ast} f } \leq
\e\norm{T_{\mathcal{ G}} } \norm{f }.$} \qedhere
\end{proof}
The following result relates an $\e$-consistent reconstruction $
f_{r}$ to the oblique projection $\pi_{\Vsb\Wsb^{\perp}}f$ of $f$.
\begin{thm}\label{T consistent reconstruction}
Let $\Hil=\mathcal{W} \oplus \mathcal{V}^{\perp}$. Let $\mathcal{F}$
be a frame for $\Wsb$, $f\in \Hil$ and $\e\geq 0 $. The following
holds:
\begin{enumerate}
\item[(i)] If $ f_{r}$ is an  $\e $-consistent reconstruction of
$f$ in $\Vsb$, then

\centerline{$\norm{f_{r} - \pi_{\Vsb\Wsb^{\perp}}f } \leq \e \norm{
\pi_{\Vsb\Wsb^{\perp}} } \norm{T_{ \mathcal{ F}}^{ \dagger}} \norm{
f }$.}

\item[(ii)] If $\norm{f_{r} - \pi_{\Vsb\Wsb^{\perp}}f } \leq \frac{
\e }{ \norm{T_{ \mathcal{ F}}}} \norm{f }$, then $ f_{r}$ is an $\e
$-consistent reconstruction of $f$ in $\Vsb$.
\end{enumerate}
\end{thm}
\begin{proof}
If $\mathcal{ G }$ is the canonical oblique dual frame of $\mathcal{
F }$ then $T_{\mathcal{ G}}=\pi_{\Vsb\Wsb^{\perp}} S_{ \mathcal{F}
}^{ \dagger }T_{ \mathcal{ F}}$. We also have, $T_{\mathcal{F}}^{
\dagger }=T_{\mathcal{F}}^{*}S_{ \mathcal{F} }^{ \dagger }$. So (i)
follows from Proposition~\ref{P cota reconstr consist dual G}.

To prove (ii), assume that $\norm{f_{r} - \pi_{\Vsb\Wsb^{\perp}}f }
\leq \frac{ \e }{ \norm{T_{ \mathcal{ F}}}} \norm{f }$. Then,
\begin{equation*}
|| T_{ \mathcal{ F}}^{*} f_{r} - T_{ \mathcal{ F}}^{*} f || = || T_{
\mathcal{ F}}^{*} f_{r} - T_{ \mathcal{ F}}^{*}
\pi_{\Vsb\Wsb^{\perp}}f || \leq \e ||f||.\qedhere
\end{equation*}
\end{proof}

The previous theorem allows to link  $\e $-consistent reconstruction
with approximate oblique duality via the next corollary.

\begin{cor}
Let $\Hil=\mathcal{W} \oplus \mathcal{V}^{\perp}$. Let $\mathcal{ F
}$ be a frame for $\Wsb$, $\mathcal{ G }$ a frame for $\Vsb$ and
$\e\geq0 $. Then
\begin{enumerate}
\item[(i)]
If $f_{r}  = T_{\mathcal{ G}} T_{ \mathcal{ F}}^{*} f $ is an
$\frac{ \e }{ || \pi_{\Vsb\Wsb^{\perp}} ||\, ||T_{ \mathcal{ F}}^{
\dagger} || }$-consistent reconstruction of $f$
 in $\Vsb$, for each $f\in\mathcal{H}$, then $\mathcal{ F
}$ and $\mathcal{ G }$ are $\e $-approximate oblique dual frames.

\item[(ii)]
If  $\mathcal{ F }$ and $\mathcal{ G }$ are $\e $-approximate
oblique dual frames, then $ f_{r}  = T_{\mathcal{ G}} T_{ \mathcal{
F}}^{*}f$ is an $\e || T_{ \mathcal{ F}} || $-consistent
reconstruction of $f$ in $\Vsb$, for each $f\in\mathcal{H}$.
\qedhere
\end{enumerate}
\end{cor}

%%%%%%%%%%%%%%%%%%%%%%%%%%%%%%%%%%%%%%%%%%%%%%%%%%%%%%%%%%%%%%
\section{Characterization of approximate oblique dual frames}

The following lemma gives equivalent conditions for approximate
oblique duality.

\begin{lem}\label{l001}
Let $\Hil=\mathcal{W} \oplus \mathcal{V}^{\perp}$. Let $\mathcal{F}
= \{f_k\}_{k=1}^{\infty}$, $\mathcal{G} = \{g_k \}_{k=1}^{\infty}$
be Bessel sequences in $\Hil $ such that $\overline{ {\rm span} }
\{f_k\}_{k=1}^{\infty} =\Wsb$ and $\overline{ {\rm span} }
\{g_k\}_{k=1}^{\infty} =\Vsb $. Let $\e \geq0 $. Then the following
assertions are equivalent:
\begin{itemize}
\item[(i)] $\{f_k
\}_{k=1}^{\infty}$ and $\{g_k\}_{k=1}^{\infty}$ are $\e$-approximate
oblique dual frames.
\item[(ii)]
$\norm{\pi_{\Wsb\Vsb^{\perp}}f - \sum_{ k=1 }^{\infty} \langle f,
g_k \rangle f_k }\leq \e\norm{f }$ for all $ f\in\Hil $.
\item[(iii)]
$\norm{\pi_{\Vsb\Wsb^{\perp}}f - \sum_{ k=1 }^{\infty} \langle f,
f_k \rangle g_k } \leq \e\norm{f }$ for all $ f\in\Hil $.
\item[(iv)]
$ | \langle\pi_{\Wsb\Vsb^{\perp}}f, g \rangle - \sum_{ k=1
}^{\infty} \langle f, g_k \rangle \langle f_k, g \rangle | \leq \e
\norm{f }\norm{g }$ for all $ f, g\in\Hil $.
\item[(v)]
$ | \langle \pi_{\Vsb\Wsb^{\perp}}f, g \rangle - \sum_{ k=1
}^{\infty} \langle f, f_k \rangle \langle g_k, g \rangle |
\leq\e\norm{f }\norm{g }$ for all $ f, g\in\Hil $.
\end{itemize}\qedhere
\end{lem}

\begin{proof}
Recall that (ii) and (iii) are the equivalent conditions (\ref{E D
1}) and (\ref{E D 2}) of Definition~\ref{D marcos duales oblicuos
aproximados}.

Let $f,g\in\Hil$. Assume that (ii) holds. Using that

\centerline{$| \langle\pi_{\Wsb\Vsb^{\perp}}f, g \rangle - \sum_{
k=1 }^{\infty} \langle f, g_k \rangle \langle f_k, g \rangle | = |
\langle\pi_{\Wsb\Vsb^{\perp}}f - \sum_{ k=1 }^{\infty} \langle f,
g_k \rangle f_k, g \rangle |$,}

\noindent and Cauchy-Schwarz inequality, we obtain (iv).

Assume now that (iv) holds. We have
\begin{align*}
\norm{\pi_{\Wsb\Vsb^{\perp}} g - \sum_{ k=1 }^{\infty} \langle g,
g_k \rangle f_k } &= \sup_{ \|f\| = 1 } | \langle
\pi_{\Wsb\Vsb^{\perp}} g - \sum_{ k=1 }^{\infty} \langle g, g_k
\rangle f_k, f \rangle | \leq \e \norm{g }.
\end{align*}
So (ii) holds.

The proof of (iii) $\Leftrightarrow$ (v) is similar to (ii)
$\Leftrightarrow$ (iv).
\end{proof}

Condition (ii) of Lemma~\ref{l001} implies
\begin{equation}\label{e006} \norm{f - \sum_{ k=1 }^{\infty}
\langle f, g_k \rangle f_k }\leq \e\norm{f }\,\,\,\text{ for }
\,f\in\Wsb,
\end{equation}
whereas (iii) implies $\norm{g - \sum_{ k=1 }^{\infty} \langle g,
f_k \rangle g_k }\leq \e\norm{g }$ for $g\in\Vsb.$ We conclude that
the elements of each subspace can be reconstructed approximately,
with a uniform relative error.

Let $\{f_k\}_{k=1}^{\infty}$ be a Bessel sequence in $\Wsb$ and $\{
g_k\}_{k=1}^{\infty}$ be a Bessel sequence in $\Vsb$. Observe that
for $ f\in\Hil $,

\centerline{$\norm{\pi_{\Wsb\Vsb^{\perp}}f - \sum_{ k=1 }^{\infty}
\langle f, g_k \rangle f_k } = \norm{\pi_{\Wsb\Vsb^{\perp}}f -
\sum_{ k=1 }^{\infty} \langle \pi_{\Wsb\Vsb^{\perp}}f, g_k \rangle
f_k }$.}

\noindent So, we have the following result:

\begin{prop}\label{P aprox subesp entonces aodf}
Let $\Hil=\mathcal{W} \oplus \mathcal{V}^{\perp}$. Let
$\{f_k\}_{k=1}^{\infty}$ be a frame for $\Wsb$ and $\{g_k
\}_{k=1}^{\infty}$ be a frame for $\Vsb$ such that (\ref{e006}) is
satisfied with $\e \geq0 $. Then $\{f_k\}_{k=1}^{\infty}$ and $\{g_k
\}_{k=1}^{\infty}$ are $\e \norm{\pi_{\Wsb\Vsb^{\perp}}
}$-approximate oblique dual frames.
\end{prop}

Similarly we have:

\begin{prop}
Let $\Hil=\mathcal{W} \oplus \Vsb^{\perp}$. Let
$\{f_k\}_{k=1}^{\infty}$ be a Bessel sequence in $\Wsb$ and $\{
g_k\}_{k=1}^{\infty}$ be a Bessel sequence in $\Hil $ such that
(\ref{e006}) holds with $\e \geq0 $. If
$\e\norm{\pi_{\Wsb\Vsb^{\perp}}} < 1$, then $\{f_k\}_{k=1}^{\infty}$
is a frame for $\Wsb$, $\{\pi_{\Vsb\Wsb^{\perp}}
g_k\}_{k=1}^{\infty}$ is a frame for $\mathcal{V}$ and
$\{f_k\}_{k=1}^{\infty}$ and $\{\pi_{\Vsb\Wsb^{\perp}}
g_k\}_{k=1}^{\infty}$ are $\e \norm{\pi_{\Wsb\Vsb^{\perp}}
}$-approximate oblique dual frames.
\end{prop}
In \cite[Lemma B.1]{Eldar-Christensen (2006)} a characterization of
oblique dual frames is given by oblique inverses. We have a similar
characterization for approximate oblique dual frames using
\emph{approximate oblique left inverses}.
\begin{defn}\label{D inversa oblicua aprox}
Let $\Hil=\mathcal{W} \oplus \mathcal{V}^{\perp}$. Let $\mathcal{F}
= \{f_k\}_{k=1}^{\infty}$ be a frame for $\Wsb$ and $\e \geq0 $. We
say that a linear bounded operator $ A : \lcuad \rightarrow\Vsb $ is
an \emph{$\e$-approximate oblique left inverse of $
T_{\mathcal{F}}^{*}$ in $\Vsb$ along $\Wsb^{\perp}$} if
$\norm{\pi_{\Vsb\Wsb^{\perp}} - A T_{\mathcal{F}}^{*} } \leq \e$.
\end{defn}

We now give the following characterization.
\begin{lem}\label{l002}
Let $\Hil=\mathcal{W} \oplus \mathcal{V}^{\perp}$. Let $\mathcal{F}
= \{f_k\}_{k=1}^{\infty}$ be a frame for $\Wsb$. Let
$\{\dl_k\}_{k=1}^{\infty}$ be the canonical basis of $\lcuad$ and
$0\leq\e<1$. The $\e$-approximate oblique dual frames of
$\mathcal{F}$ in $\Vsb$ are the sequences $\{g_k\}_{k=1}^{\infty} =
\{A\dl_k\}_{k=1}^{\infty}$, where $A : \lcuad \longrightarrow\Vsb$
is an $\e$-approximate oblique left inverse of $
T_{\mathcal{F}}^{*}$ in $\Vsb$ along $\Wsb^{\perp}$. \qedhere
\end{lem}
\begin{proof}
Clearly, if $\mathcal{G} = \{g_k\}_{k=1}^{\infty}$ is an
$\e$-approximate oblique dual frame of $\{f_k\}_{k=1}^{\infty}$ in
$\Vsb$, then $T_{\mathcal{G}}$ is an $\e$-approximate oblique left
inverse of $T_{\mathcal{F}}^{*}$ in $\Vsb$ along $\Wsb^{\perp}$ and
$\{g_k\}_{k=1}^{\infty} = \{T_{\mathcal{G}}\dl_k\}_{k=1}^{\infty}$.

Let now $ A : \lcuad \longrightarrow\Vsb $ be an $\e$-approximate
oblique left inverse. Let $\mathcal{G} = \{g_k\}_{k=1}^{\infty} =
\{A \dl_k\}_{k=1}^{\infty}$. By \cite[Theorem 3.2.3]{Christensen
(2016)}, $\mathcal{G}$ is a Bessel sequence in $\Vsb$ with
$T_{\mathcal{G}}=A$. We have $\norm{\pi_{\Vsb\Wsb^{\perp}} -
T_{\mathcal{G}} T_{\mathcal{F}}^{*} } \leq \e. $ By Remark~\ref{Obs
definicion de mdoa}(iii), $\mathcal{G}$ is a frame for $\Vsb$. Thus,
$\mathcal{G}$ is an $\e$-approximate oblique dual frame of $\{f_k
\}_{k=1}^{\infty}$ in $\Vsb$.
\end{proof}

We note that $ A:\lcuad\longrightarrow\Vsb $ is an $\e$-approximate
oblique left inverse of $ T_{ \mathcal{ F}}^{\ast}$ in $\Vsb$ if and
only if $ A = L + R $ where $ L:\lcuad\longrightarrow\Vsb $ is an
oblique left inverse of $ T_{ \mathcal{ F}}^{\ast}$ in $\Vsb$ and $
R:\lcuad\longrightarrow\Vsb $ is a bounded operator such that
$\norm{RT_{ \mathcal{F} }^{\ast}} \leq \e$. So, from
Lemma~\ref{l002} and \cite[Lemma B.2]{Christensen-Eldar (2004)}, we
obtain a description of the approximate oblique dual frames:
\begin{thm}
Let $\Hil=\mathcal{W} \oplus \mathcal{V}^{\perp}$. Let $\mathcal{F}
= \{f_k\}_{k=1}^{\infty}$ be a frame for $\Wsb$ and $0\leq\e<1$.
Then any $\e$-approximate oblique dual frame $\mathcal{G}=
\{g_k\}_{k=1}^{\infty} \subset\Vsb $ of $\mathcal{F}$ in $\Vsb$ is
of the form

\centerline{$\{g_k\}_{k=1}^{\infty} = \{\pi_{\Vsb\Wsb^{\perp}} S_{
\mathcal{F} }^{ \dagger } f_k + h_k - \sum_{j=1}^{\infty} \langle
S_{ \mathcal{F} }^{ \dagger }f_k, f_j \rangle h_j +
r_{k}\}_{k=1}^{\infty}$}

\noindent where $\{h_k\}_{k=1}^{\infty}$ and $\{r_k
\}_{k=1}^{\infty}$ are Bessel sequences in $\Vsb$ and $ || \sum
\langle f,f_{k} \rangle r_{k} || \leq \e ||f|| $ for all $f\in
\mathcal{H}$.
\end{thm}

%%%%%%%%%%%%%%%%%%%%%%%%%%%%%%%%%%%
\section{Improving approximation}

Given a frame $\mathcal{F} = \{f_k\}_{k=1}^{\infty}$ in $\Wsb$ and
an approximate oblique dual frame $\mathcal{G} = \{g_k
\}_{k=1}^{\infty} \subset\Vsb $ of it, the following proposition
says that it is possible to construct from $\mathcal{G}$ other
approximate oblique dual frames for $\mathcal{F}$ in $\Vsb$ that
provide a reconstruction as ``good" as needed. This is analogous to
a result for approximate dual frames \cite[Proposition
3.2]{Christensen-Laugesen (2010)}.
\begin{prop}\label{p001}
Let $\Hil=\mathcal{W} \oplus \mathcal{V}^{\perp}$. Let $\mathcal{F}
= \{f_k\}_{k=1}^{\infty} \subset\Wsb$ and $\mathcal{G} =
\{g_k\}_{k=1}^{\infty} \subset\Vsb $ be $\e$-approximate oblique
dual frames with $0 \leq \e < 1$. For a fixed $ N\in\N $, let
\begin{equation}\label{e004}
\widetilde{g}_k^{(N)} =\sum_{ n=0 }^{ N } (\pi_{\Vsb\Wsb^{\bot}} -
T_{\mathcal{G}} T_{\mathcal{F}}^{\ast})^n g_k
\end{equation}
and $\mathcal{\widetilde{G}}_{N} = \{\widetilde{g}_k^{(N)}
\}_{k=1}^{\infty}$. Then $\mathcal{\widetilde{G}}_{N}$ is an
$\e^{N+1}$-approximate oblique dual frame of $\mathcal{F}$ in $\Vsb$
and
\begin{equation}\label{e005}
\norm{\pi_{\Vsb\Wsb^{\bot}} - T_{\mathcal{\widetilde{G}}_{N}}
T_{\mathcal{F}}^{\ast}} \longrightarrow 0 \quad as \quad
N\rightarrow \infty.
\end{equation}
\end{prop}
\begin{proof}
Let $\widetilde{g}_k^{(N)} = \sum_{ n=0 }^{ N }
(\pi_{\Vsb\Wsb^{\bot}} - T_{\mathcal{G}} T_{\mathcal{F}}^{\ast})^n
g_k $ for $ N\in\N $ and $\mathcal{\widetilde{G}}_{N} = \{
\widetilde{g}_k^{(N)}\}_{k=1}^{\infty}$.

Since $\{g_k\}_{k=1}^{\infty}$ is a Bessel sequence and $\pi_{\Vsb
\Wsb^{\bot}} - T_{\mathcal{G}} T_{\mathcal{F}}^{\ast}$ is a bounded
operator, it follows that $\{\widetilde{g}_k^{(N)}
\}_{k=1}^{\infty}$ is a Bessel sequence. Hence, by \cite[Theorem
3.2.3]{Christensen (2016)} its synthesis operator $
T_{\mathcal{\widetilde{G}}_{N}}$ is well defined. Let $ f\in\Hil $,
then
\begin{align*}
T_{\mathcal{\widetilde{G}}_{N}} T_{\mathcal{F}}^{\ast} f =&
\sum_{k=1}^{\infty} \langle f, f_k \rangle \widetilde{g}_k^{(N)} =
\sum_{k=1}^{\infty} \langle f, f_k \rangle \sum_{ n=0 }^{ N }
(\pi_{\Vsb\Wsb^{\bot}} - T_{\mathcal{G}} T_{\mathcal{F}}^{\ast})^n g_k\\
=& \sum_{ n=0 }^{ N } (\pi_{\Vsb\Wsb^{\bot}} - T_{\mathcal{G}}
T_{\mathcal{F}}^{\ast})^n T_{\mathcal{G}}
T_{\mathcal{F}}^{\ast} f\\
=& \sum_{ n=0 }^{ N } (\pi_{\Vsb\Wsb^{\bot}} - T_{\mathcal{G}}
T_{\mathcal{F}}^{\ast})^n (\pi_{\Vsb\Wsb^{\bot}} - (\pi_{\Vsb
\Wsb^{\bot}} -
T_{\mathcal{G}} T_{\mathcal{F}}^{\ast})) \pi_{\Vsb\Wsb^{\bot}} f\\
=& \sum_{ n=0 }^{ N } (\pi_{\Vsb\Wsb^{\bot}} - T_{\mathcal{G}}
T_{\mathcal{F}}^{\ast})^n \pi_{\Vsb\Wsb^{\bot}} f - \sum_{ n=0 }^{ N
} (\pi_{\Vsb\Wsb^{\bot}} - T_{\mathcal{G}} T_{\mathcal{F}}^{\ast}
)^{ n+1 }
\pi_{\Vsb\Wsb^{\bot}} f \\
=& \pi_{\Vsb\Wsb^{\bot}} f - (\pi_{\Vsb\Wsb^{\bot}} -
T_{\mathcal{G}} T_{\mathcal{F}}^{\ast})^{ N+1 }
\pi_{\Vsb\Wsb^{\bot}} f
\end{align*}
Now, using that $\Wsb^{ \bot } \subseteq \mathcal{N}(\pi_{\Vsb
\Wsb^{\bot}} - T_{\mathcal{G}} T_{\mathcal{F}}^{\ast}) $,
\begin{equation*}
\pi_{\Vsb\Wsb^{\bot}} f - T_{\mathcal{\widetilde{G}}_{N}}
T_{\mathcal{F}}^{\ast} f = (\pi_{\Vsb\Wsb^{\bot}} - T_{\mathcal{G}}
T_{\mathcal{F}}^{\ast})^{ N+1 } \pi_{\Vsb\Wsb^{\bot}} f = (
\pi_{\Vsb\Wsb^{\bot}} - T_{\mathcal{G}} T_{\mathcal{F}}^{\ast})^{
N+1 } f
\end{equation*}
Therefore, $\pi_{\Vsb\Wsb^{\bot}} - T_{\mathcal{\widetilde{G}}_{N}}
T_{\mathcal{F}}^{\ast} = (\pi_{\Vsb\Wsb^{\bot}} - T_{\mathcal{G}}
T_{\mathcal{F}}^{\ast})^{ N+1 }$ and
\begin{equation}\label{e007}
\norm{\pi_{\Vsb\Wsb^{\bot}} - T_{\mathcal{\widetilde{G}}_{N}}
T_{\mathcal{F}}^{\ast}} \leq \norm{\pi_{\Vsb\Wsb^{\bot}} -
T_{\mathcal{G}} T_{\mathcal{F}}^{\ast}}^{ N+1 }  \leq\e^{N+1}
\rightarrow 0
\end{equation}
as $ N\rightarrow \infty. $\qedhere
\end{proof}
\begin{rem}
Note that in the previous proof we obtain

\centerline{$\pi_{\Vsb\Wsb^{\bot}} - T_{\mathcal{\widetilde{G}}_{N}}
T_{\mathcal{F}}^{\ast} = (\pi_{ \Vsb\Wsb^{\bot}} - T_{\mathcal{G}}
T_{\mathcal{F}}^{\ast})^{ N+1 }$.}

Considering
\begin{equation}\label{E LN}
L_N = \sum_{ n=0}^{ N } (\pi_{\Vsb\Wsb^{\bot}} - T_{\mathcal{G}}
T_{\mathcal{F}}^{\ast})^n,
\end{equation}
the operator $ T_{ \mathcal{\widetilde{G}}_{N} }$ can be expressed
as $ T_{ \mathcal{\widetilde{G}}_{N} } = L_N T_{\mathcal{G}}$.

Setting $\widetilde{f}_k^{(N)}=\sum_{ n=0 }^{ N }
(\pi_{\Wsb\Vsb^{\bot}} - T_{\mathcal{F}}T_{\mathcal{G}}^{\ast})^n
f_k$ and $\mathcal{\widetilde{F}}_{N} = \{\widetilde{f}_k^{(N)}
\}_{k=1}^{\infty}$, we have $ T_{ \mathcal{\widetilde{F}}_{N} } =
L_N^* T_{\mathcal{F}}$, $\pi_{\Vsb\Wsb^{\bot}} - T_{\mathcal{G}}
T_{\mathcal{\widetilde{F}}_{N}}^{\ast} = (\pi_{\Vsb\Wsb^{\bot}} -
T_{\mathcal{G}} T_{\mathcal{F}}^{\ast})^{ N+1 }$ and
\begin{equation}\label{e009}
\norm{\pi_{\Vsb\Wsb^{\bot}} - T_{\mathcal{G}}
T_{\mathcal{\widetilde{F}}_{N}}^{\ast}} \leq
\norm{\pi_{\Vsb\Wsb^{\bot}} - T_{\mathcal{G}}
T_{\mathcal{F}}^{\ast}}^{ N+1 } \leq\e^{N+1} \rightarrow 0. \newline
\end{equation}
From (\ref{e007}) and (\ref{e009}), we observe that we can improve
the approximation given by the approximate oblique dual frames
$\mathcal{F}$ and $\mathcal{G}$, modifying either of them.
\end{rem}

The following proposition shows that the frames
$\mathcal{\widetilde{F}}_{N}$ and $\mathcal{\widetilde{G}}_{N}$
defined previously are related with oblique dual frames of
$\mathcal{F}$ and $\mathcal{G}$.

\begin{prop}
Let $\Hil=\mathcal{W} \oplus \mathcal{V}^{\perp}$. Let $\mathcal{F}
= \{f_k\}_{k=1}^{\infty} \subset\Wsb$ and $\mathcal{G} =
\{g_k\}_{k=1}^{\infty} \subset\Vsb $ be $\e$-approximate oblique
dual frames with $0 \leq \e < 1$. Then $ I - (\pi_{\Vsb\Wsb^{\bot}}
- T_{\mathcal{G}} T_{\mathcal{F}}^{\ast} ) $ is invertible and

\centerline{$L := (I - (\pi_{\Vsb\Wsb^{\bot}} - T_{\mathcal{G}}
T_{\mathcal{F}}^{\ast}))^{-1} = I + \sum_{n=1}^{\infty} (
\pi_{\Vsb\Wsb^{\bot}} - T_{\mathcal{G}} T_{\mathcal{F}}^{\ast}
)^n.$}

\noindent Let $\mathcal{\widetilde{G}} = \{\widetilde{g}_k
\}_{k=1}^{\infty}$ with $\widetilde{g}_k = LT_{\mathcal{G}} \delta_k
$ and $\mathcal{\widetilde{F}} = \{\widetilde{f}_k
\}_{k=1}^{\infty}$ with $\widetilde{f}_k = L^{\ast}
T_{\mathcal{F}}\delta_k $. Then:
\begin{enumerate}
  \item[(i)] $\mathcal{\widetilde{G}}$ is an oblique dual frame
of $\mathcal{F}$ in $\Vsb$.
  \item[(ii)] $\mathcal{\widetilde{F}}$ is an oblique dual frame of $\mathcal{G}$
in $\Wsb$.
  \item[(iii)] $\mathcal{\widetilde{G}}$ is an
$\frac{\e}{1-\e}$-approximate oblique dual frame of
$\mathcal{\widetilde{F}}$ in $\Vsb$.
\end{enumerate}
\end{prop}
\begin{proof}
The assertions about $ I - (\pi_{\Vsb\Wsb^{\bot}} - T_{\mathcal{G}}
T_{\mathcal{F}}^{\ast}) $ follow from Neumann's Theorem.

If $L_N$ is as in (\ref{E LN}), then $ L_N \rightarrow L $ as $ N
\rightarrow \infty $. So, $T_{ \mathcal{\widetilde{G}}_N }=L_N
T_{\mathcal{G}} \rightarrow T_{ \mathcal{\widetilde{G}} }=L
T_{\mathcal{G}}$ as $ N \rightarrow \infty $. Using (\ref{e007}), we
have

\centerline{$||\pi_{\Vsb\Wsb^{\bot}}-T_{ \mathcal{\widetilde{G}} }
T_{\mathcal{F}}^{\ast} || = ||\pi_{\Vsb\Wsb^{\bot}}-T_{
\mathcal{\widetilde{G}}_N } T_{\mathcal{F}}^{\ast} || +||T_{
\mathcal{\widetilde{G}}_N } T_{\mathcal{F}}^{\ast}-T_{ \mathcal{\widetilde{G}} } T_{\mathcal{F}}^{\ast}||\\
\leq \e^{ N+1 } +||T_{ \mathcal{\widetilde{G}}_N }-T_{
\mathcal{\widetilde{G}} }||\, || T_{\mathcal{F}}||$}

\noindent Since the right hand side tends to $0$ as $ N \rightarrow
\infty $, $T_{ \mathcal{\widetilde{G}} }
T_{\mathcal{F}}^{\ast}=\pi_{\Vsb\Wsb^{\bot}}$. This shows that
$\mathcal{\widetilde{G}}$ is an oblique dual frame of $\mathcal{F}$
in $\Vsb$. Similarly, it can be proved that
$\mathcal{\widetilde{F}}$ is an oblique dual frame of $\mathcal{G}$
in $\Wsb$. Since $\mathcal{R}(I-L)\subseteq\Vsb$, by (i),

\centerline{$\pi_{\Vsb\Wsb^{\bot}}-T_{ \mathcal{\widetilde{G}} } T_{
\widetilde{\mathcal{F}} }^{\ast} =\pi_{\Vsb
\Wsb^{\bot}}-T_{\widetilde{\mathcal{G}}}T_{\mathcal{F}}^*L =
\pi_{\Vsb\Wsb^{\bot}}(I-L)=
\sum_{n=1}^{\infty}(\pi_{\Vsb\Wsb^{\bot}} - T_{\mathcal{G}}
T_{\mathcal{F}}^{\ast})^n.$} \noindent Then
$||\pi_{\Vsb\Wsb^{\bot}}-T_{ \mathcal{\widetilde{G}} } T_{
\widetilde{\mathcal{F}} }^{\ast}|| \leq \frac{\e}{1-\e}$.
\end{proof}
\begin{rem}\label{R Fmonio Gmonio no duales}
We note that $\mathcal{\widetilde{G}}$ is an oblique dual frame of
$\mathcal{\widetilde{F}}$ in $\Vsb$ if and only if $\mathcal{G}$ is
an oblique dual frame of $\mathcal{F}$ in $\Vsb$. In this case,
$\mathcal{\widetilde{G}}=\mathcal{G}$ and
$\mathcal{\widetilde{F}}=\mathcal{G}$. In effect, assume that
$\mathcal{\widetilde{G}}$ is an oblique dual frame of
$\mathcal{\widetilde{F}}$ in $\Vsb$. Then, from the equalities above
$L=I$. Conversely, if $\mathcal{G}$ is an oblique dual frame of
$\mathcal{F}$ in $\Vsb$, then $\e=0$ and, from the definition of
$L$, $L=I$. In both cases, $\mathcal{\widetilde{G}}=\mathcal{G}$ and
$\mathcal{\widetilde{F}}=\mathcal{G}$.
\end{rem}
\begin{rem}
It is natural to ask if we can have an approximation as close as
wanted to $\pi_{\Vsb\Wsb^{\bot}}f$ using simultaneously the frames
$\mathcal{\widetilde{F}}_{N}$ and $\mathcal{\widetilde{G}}_{N}$. The
answer is no. More precisely, since

\centerline{$||\pi_{\Vsb\Wsb^{\bot}}-T_{ \mathcal{\widetilde{G}} }
T_{ \widetilde{\mathcal{F}} }^{\ast}|| \leq ||\pi_{\Vsb
\Wsb^{\bot}}-T_{ \mathcal{\widetilde{G}}_{N} } T_{
\widetilde{\mathcal{F}}_{N} }^{\ast}||+||T_{
\mathcal{\widetilde{G}}_{N} } T_{ \widetilde{\mathcal{F}}_{N}
}^{\ast}-T_{ \mathcal{\widetilde{G}} } T_{ \widetilde{\mathcal{F}}
}^{\ast}||$}

\noindent and

\centerline{$\lim_{N \rightarrow \infty}||T_{
\mathcal{\widetilde{G}}_{N} } T_{ \widetilde{\mathcal{F}}_{N}
}^{\ast}-T_{ \mathcal{\widetilde{G}} } T_{ \widetilde{\mathcal{F}}
}^{\ast}||=0$,}

\noindent we obtain, using Remark~\ref{R Fmonio Gmonio no duales},

\centerline{$\liminf_{N\rightarrow\infty}||\pi_{\Vsb
\Wsb^{\bot}}-T_{ \mathcal{\widetilde{G}}_{N}} T_{
\widetilde{\mathcal{F}}_{N} }^{\ast}||\geq||\pi_{\Vsb
\Wsb^{\bot}}-T_{ \mathcal{\widetilde{G}} } T_{
\widetilde{\mathcal{F}} }^{\ast}||>0$.}
\end{rem}

%%%%%%%%%%%%%%%%%%%%%%%%%%%%%%%%%%%%%%%%%%%%%%%%%%%%%
\section{Approximate oblique dual frames by perturbation}

The following propositions say that if two frames are ``close'',
each oblique dual frame of one of them is an approximate oblique
dual frame of the other. They are extensions  of \cite[Proposition
4.1 and Proposition 4.3]{Christensen-Laugesen (2010)}, which are
results about approximate dual frames.
\begin{prop}
Let $\Hil=\mathcal{W} \oplus \mathcal{V}^{\perp}$. Let $0 \leq \e <
1$, $\mathcal{\widetilde{G}} = \{\widetilde{g}_{k}\}_{k=1}^{\infty}$
be a frame for $\Vsb$, $\mathcal{G} = \{g_k\}_{k=1}^{\infty}
\subset\Vsb $ a sequence such that $\sum_{k=1}^{\infty} | \langle f,
g_k - \widetilde{g}_{k} \rangle|^2 \leq r \norm{f }^2 $ for each $
f\in\Vsb $, and $\mathcal{F} = \{f_k\}_{k=1}^{\infty}$ an oblique
dual frame of $\mathcal{\widetilde{G}}$ in $\Wsb$ whose upper frame
bound $\beta $ satisfies \mbox{$\sqrt{\beta r} \leq \e$}. Then,
$\mathcal{G}$ is an $\e$-approximate dual frame of $\mathcal{F}$ in
$\Vsb$.
\end{prop}
\begin{proof}
We will first see that $\{g_k\}_{k=1}^{\infty}$ is a Bessel
sequence. Let $f\in\Vsb $. Applying the Cauchy-Schwarz inequality,
\begin{align}\label{besprop}
\sum_{k=1}^{\infty} | \langle f, g_k \rangle|^2 \leq&
\sum_{k=1}^{\infty} \left(| \langle f, g_k - \widetilde{g}_{k}
\rangle | +
| \langle f, \widetilde{g}_{k} \rangle | \right)^2 \nonumber\\
=& \sum_{k=1}^{\infty} | \langle f, g_k - \widetilde{g}_{k}
\rangle|^2 + 2\sum_{k=1}^{\infty} | \langle f, g_k -
\widetilde{g}_{k} \rangle | | \langle f, \widetilde{g}_{k} \rangle |
+
\sum_{k=1}^{\infty} | \langle f, \widetilde{g}_{k} \rangle|^2 \nonumber\\
\leq& (r + 2\sqrt{ r\widetilde{\beta} } + \widetilde{\beta}) \norm{f
}^2,
\end{align}
where $\widetilde{\beta}$ is the upper frame bound for $\{
\widetilde{g}_{k}\}_{k=1}^{\infty}$. Hence, $\{g_k
\}_{k=1}^{\infty}$ is a Bessel sequence with bound $ (\sqrt{ r } +
\sqrt{ \widetilde{\beta} })^2 $.

Note that the condition given for $\{g_k\}_{k=1}^{\infty}$ can be
written as $\norm{T_{\mathcal{G}}^{\ast} -
T_{\mathcal{\widetilde{G}}}^{\ast}} \leq\sqrt{ r }. $

We have
$$\norm{\pi_{\Wsb\Vsb^{\perp}} -
T_{\mathcal{F}} T_{\mathcal{G}}^{\ast}} = \norm{T_{ \mathcal{ F}}
T_{ \mathcal{\widetilde{G}} }^{\ast} - T_{ \mathcal{ F}}
T_{\mathcal{ G}}^{\ast}} \leq \norm{T_{\mathcal{F}} } \norm{
T_{\mathcal{\widetilde{G}}}^{\ast} - T_{\mathcal{G}}^{\ast}} \leq
\sqrt{\beta r} \leq \e. $$

Therefore, by Definition~\ref{D marcos duales oblicuos aproximados}
and Remark~\ref{Obs definicion de mdoa}(iii),
$\{f_k\}_{k=1}^{\infty}$ and $\{g_k\}_{k=1}^{\infty}$ are
$\e$-approximate oblique dual frames.
\end{proof}
\begin{prop}
Let $\Hil=\mathcal{W} \oplus \mathcal{V}^{\perp}$. Let $\e\geq0$,
$\mathcal{F} = \{f_k\}_{k=1}^{\infty}$ be a frame for $\Wsb$ with
bounds $\alpha $ and $\beta $ and $\mathcal{\widetilde{G}} =
\{\widetilde{g}_{k}\}_{k=1}^{\infty} \subset\Wsb$ such that \mbox{
$\sum_{k=1}^{\infty} | \langle f, f_k - \widetilde{g}_{k} \rangle|^2
\leq r \norm{f }^2 $ }, for $f \in\Wsb$, with $ r \leq \frac{
\alpha\e^{2} }{\left(\norm{ \pi_{\Vsb\Wsb^{\perp}} } + \e\right)^2
}$. Then,
\begin{itemize}
\item[(i)]
$\mathcal{\widetilde{G}}$ is a frame for $\Wsb$ with bounds $ (
\sqrt{ \alpha } -\sqrt{ r })^2 $ and $ (\sqrt{ r } +\sqrt{ \beta }
)^2 $.

\item[(ii)]

If $\mathcal{G} = \{g_k\}_{k=1}^{\infty} \subset\Vsb $ is given by $
g_k = \pi_{\Vsb\Wsb^{\perp}} (
T_{\mathcal{\widetilde{G}}}T_{\mathcal{\widetilde{G}}}^{\ast})^{
\dag } \widetilde{g}_{k}$, then $\mathcal{G}$ is an $\e$-approximate
oblique dual frame of $\mathcal{F}$ in $\Vsb$. \qedhere
\end{itemize}
\end{prop}
\begin{proof}
(i) Since $r < \alpha$, this part follows from \cite[Corollary
22.1.5]{Christensen (2016)}.

To see (ii), observe that, by \cite[Lemma 5.1.5]{Christensen
(2016)}, the canonical dual frame $\{(
T_{\mathcal{\widetilde{G}}}T_{\mathcal{\widetilde{G}}}^{\ast})^{
\dag } \widetilde{g}_{k}\}_{k=1}^{\infty}$ of $\{\widetilde{g}_{k}
\}_{k=1}^{\infty}$ in $\Wsb$ has bounds $\frac{1}{ (\sqrt{ r } +
\sqrt{ \beta })^2 }$ and $\frac{1}{ (\sqrt{ \alpha } -\sqrt{ r } )^2
}$.

By \mbox{\cite[Theorem 3.2]{Christensen-Eldar (2004)}}, the sequence
$\mathcal{G} = \{g_k\}_{k=1}^{\infty} \subset\Vsb $ given by $ g_k =
\pi_{\Vsb\Wsb^{\perp}} (
T_{\mathcal{\widetilde{G}}}T_{\mathcal{\widetilde{G}}}^{\ast})^{
\dag } \widetilde{g}_{k}$ is an oblique dual frame of
$\mathcal{\widetilde{G}}$ in $\Vsb$. Its synthesis operator
satisfies $\norm{T_{\mathcal{G}}} \leq \frac{ \norm{
\pi_{\Vsb\Wsb^{\perp}}}}{\sqrt{\alpha} -\sqrt{r} }. $ Hence,

\centerline {$\norm{\pi_{\Vsb\Wsb^{\perp}} - T_{\mathcal{ G}} T_{
\mathcal{ F}}^{\ast}} = \norm{T_{\mathcal{ G}}
T_{\mathcal{\widetilde{G}} }^{\ast} - T_{\mathcal{ G}}T_{ \mathcal{
F}}^{\ast}} \leq
\norm{T_{\mathcal{G}}}\norm{T_{\mathcal{\widetilde{G}}}^{\ast} -
T_{\mathcal{F}}^{\ast}} \leq \frac{\norm{\pi_{\Vsb\Wsb^{\perp}}}}{
\sqrt{\alpha} -\sqrt{r}}\sqrt{r} \leq \e. $}

So $\{f_k\}_{k=1}^{\infty}$ and $\{g_k\}_{k=1}^{\infty}$ are
$\e$-approximate oblique dual frames. \qedhere
\end{proof}

\section{Approximate reconstruction in shift-invariant subspaces}

In this section we give conditions on the generators of
shift-invariant subspaces $\Wsb$ and $\Vsb$ of $\mathcal{H}=
\Lebcuad$, in order to obtain an approximate reconstruction in one
of them. We emphasize that for these results the subspaces don't
necessarily decompose $\Lebcuad$ in direct sum, as we assumed
before. As a consequence we obtain in the next section sufficient
conditions on the generators for approximate oblique duality (see
Corollaries \ref{C cond suf doa} and \ref{C cond suf doaH}). In
order to prove the results we will use this lemma:
\begin{lem}\label{L saca f afuera}
Let $\phi, \widetilde{\phi}\in\Lebcuad$ be such that
$\{T_k\phi\}_{k\in\Z}$ and $\{T_k\widetilde{\phi}\}_{k\in\Z}$ are
Bessel sequences. The following holds:
\begin{enumerate}
  \item[(i)] If $f\in\Lebcuad$, then $[\,\widehat{f},\widehat{\widetilde{\phi}}\,]\in \mathcal{C}_{{\rm per}}^{2}$ and

\centerline{$\sum_{k\in\Z}\langle\, \widehat{f},
\widehat{T_k\widetilde{\phi}} \,\rangle \widehat{T_k\phi} =
\widehat{\phi}\,[\,\widehat{f},\widehat{\widetilde{\phi}}\,]$.}

  \item[(ii)] If $ f\in\overline{\mathrm{{\rm span}}}\{T_k\phi\}_{k\in\Z}$,
then
\begin{equation}\label{E L saca f afuera}
\sum_{k\in\Z}\langle\, \widehat{f}, \widehat{T_k\widetilde{\phi}}
\,\rangle
\widehat{T_k\phi}=\widehat{f}\,[\,\widehat{\phi},\widehat{\widetilde{\phi}}\,].
\end{equation}
\end{enumerate}
\end{lem}
\begin{proof}
By the Cauchy-Schwarz inequality and Theorem~\ref{T Benedetto}(i),
if $f\in\Lebcuad$ then

\centerline{$\int_{0}^{1}|[\,\widehat{f},\widehat{\widetilde{\phi}}\,](\ga)|^{2}d\ga
\leq \widetilde{\beta} \int_{0}^{1}\sum_{n\in\Z}
|\widehat{f}(\ga+n)|^2d\ga\leq \widetilde{\beta} ||f||^{2},$}
\noindent where $\widetilde{\beta}$ is an upper frame bound of
$\{T_k\widetilde{\phi}\}_{k\in\Z}$. So the equality in (i) -- which
coincides with \cite[Equality (12)]{Christensen-Eldar (2004)} --
holds.

From (i), if $f=\sum_{k\in\Z} c_k T_k\phi$, where
$\{c_k\}_{k\in\Z}\in \ell^2(\Z) $ has a finite number of nonzero
elements,
\begin{align*}
\sum_{k\in\Z}\langle\, \widehat{f}, \widehat{T_k\widetilde{\phi}}
\,\rangle \widehat{T_k\phi} (\ga) &= \widehat{\phi}(\ga)
\sum_{n\in\Z}\left(\sum_{k\in\Z} c_k e^{-i2\pi k \ga}\right)
\widehat{\phi}(\ga+n)\overline{\widehat{\widetilde{\phi}}(\ga+n)}\\
&=\left(\sum_{k\in\Z} c_k e^{-i2\pi k \ga}\right)
\widehat{\phi}(\ga)\,[\,\widehat{\phi},\widehat{\widetilde{\phi}}\,](\ga)=\widehat{f}(\ga)\,[\,\widehat{\phi},\widehat{\widetilde{\phi}}\,](\ga).
\end{align*}
It follows that (\ref{E L saca f afuera}) holds for $ f\in
\mathrm{{\rm span}}\{T_k\phi\}_{k\in\Z}$. Since the operator $f
\mapsto \sum_{k\in\Z}\langle\, f, T_k\widetilde{\phi} \,\rangle
T_k\phi$ is continuous and by (\ref{E cota superior BB})
$[\,\widehat{\phi},\widehat{\widetilde{\phi}}\,]\in
L^{\infty}(\mathbb{R})$, (\ref{E L saca f afuera}) holds for
$f\in\overline{\mathrm{{\rm span}}}\{T_k\phi\}_{k\in\Z}$.
\end{proof}

The key result of this section is the following:

\begin{thm}\label{caractdualestrasla}
Let $\phi, \widetilde{\phi}\in\Lebcuad$ be such that
$\{T_k\phi\}_{k\in\Z}$ and $\{T_k\widetilde{\phi}\}_{k\in\Z}$ are
Bessel sequences. Let $\Wsb =\overline{\mathrm{{\rm
span}}}\{T_k\phi\}_{k\in\Z}$ and $\Vsb =\overline{{\rm
span}}\{T_k\widetilde{\phi}\}_{k\in\Z}$. Then, given $\e\geq0 $, the
following statements are equivalent:
\begin{itemize}
\item[(i)]
$\| f - \sum_{k\in\Z}\langle f, T_k\widetilde{\phi} \rangle T_k\phi
\|\leq\e\|f\| $, for $f\in\Wsb$.
\item[(ii)] $||f - \sum_{k\in\Z}\langle f, T_k P_{\Wsb}
\widetilde{\phi} \rangle T_k\phi || \leq\e\|f\| $, for $f\in\Wsb$.
\item[(iii)]
$\abs{[\,\widehat{\phi},\widehat{\widetilde{\phi}}\,](\ga) - 1
}\leq\e $ for a.e. $\ga\in N(\Phi)^{c}$.
\end{itemize}
Moreover, if $\e < 1 $, then statements (i) to (iii) are also
equivalent to:
\begin{itemize}
\item[(iv)]$\{T_k\phi\}_{k\in\Z}$ and $\{T_k P_{\Wsb}\widetilde{\phi}
\}_{k\in\Z}$ are $\e$-approximate dual frames in $\Wsb$.
\end{itemize}
\end{thm}
\begin{proof}
(i) $\Leftrightarrow$ (ii): The operators $ P_{\Wsb}$ and $ T_k $
commute. So, given $f\in\Wsb$,

\centerline{$||f - \sum_{k\in\Z}\langle f, T_k P_{\Wsb}
\widetilde{\phi} \rangle T_k\phi ||=|| f - \sum_{k\in\Z}\langle f,
T_k\widetilde{\phi} \rangle T_k\phi ||.$ }

(i) $\Rightarrow$ (iii): Assume that (iii) does not hold. Then there
exists $E\subseteq N(\Phi)^{c}\cap [0,1)$ such that $|E|>0 $ and

\centerline{$|[\,\widehat{\phi},\widehat{\widetilde{\phi}}\,](\ga) -
1 |>\e $ for $\ga\in E$.}

\noindent We will see that there exists $f\in\Wsb$ such that

\centerline{$||f - \sum_{k\in\Z}\langle f, T_k\widetilde{\phi}
\rangle T_k\phi ||>\e ||f|| $.}

We can write $E = \bigcup_{ k\in\N } E_k $, where

\centerline{$ E_k =\left\{\ga\in N(\Phi)^{c}\cap [0,1): |
[\,\widehat{\phi},\widehat{\widetilde{\phi}}\,](\ga) - 1 |
\geq\frac{1}{k} + \e\right\}$}

\noindent for $ k\in\N $. If $ | E_k | = 0 $ for all $ k $, then $
|E| = 0 $, which is a contradiction. So, there exists $\e'
>\e $ and $ E' \subseteq E $ such that $ | E' |> 0 $ and
\begin{equation}\label{condcontrarecip}
|[\,\widehat{\phi},\widehat{\widetilde{\phi}}\,](\ga) - 1 |\geq\e'
\textrm{  for  }\ga\in E'.
\end{equation}
\noindent Let $E'_{p}=\bigcup_{k\in\Z}(E'+k)$ and $F \in
\mathcal{C}_{{\rm per}}^{2}$ such that $\textrm{supp}(F) \subseteq
E'_{p}$ and $\abs{\textrm{supp}(F)} >0$. Let $f\in\Wsb$ that
verifies $\widehat{f} = F\widehat{\phi}$. Then $f\neq0$ and by
Lemma~\ref{L saca f afuera}(ii) and (\ref{condcontrarecip}),
\begin{align*}
|| f - \sum_{k\in\Z}\langle f, T_k\widetilde{\phi} \rangle T_k\phi
||^2 =&\left\|\widehat{f}\left(1 -
[\,\widehat{\phi},\widehat{\widetilde{\phi}}\,]\right)\right\|^2=\left\|\chi_{E'_{p}}\widehat{ f }\left(1 - [\,\widehat{\phi},\widehat{\widetilde{\phi}}\,]\right)\right\|^{2} \\
\geq& (\e')^2 ||\chi_{E'_{p}}\widehat{ f } ||^2 = (\e')^2 ||
\widehat{f} ||^2>\e^2 ||f||^2,
\end{align*}
which contradicts (i).

(iii) $\Rightarrow$ (i): Let $f=\sum_{k\in\Z} c_k T_k\phi$, where
$\{c_k\}_{k\in\Z}\in \ell^2(\Z) $ has a finite number of nonzero
elements. Note that if $\Phi(\ga) = 0 $ then $\widehat{f}(\ga) = 0$.
By Lemma~\ref{L saca f afuera}(ii) and (iii), we have that

\centerline{$||f - \sum_{k\in\Z}\langle f, T_k\widetilde{\phi}
\rangle T_k\phi || =\left\| \chi_{ N(\Phi)^{c} }\widehat{f}
\left(1-[\,\widehat{\phi},\widehat{\widetilde{\phi}}\,]\right)\right\|
\leq \e \| \widehat{f} \| = \e \|f\|.$} So, as in the proof of
Lemma~\ref{L saca f afuera}(ii), (i) holds for $f\in
\overline{\mathrm{{\rm span}}}\{T_k\phi\}_{k\in\Z}$.

If $\e < 1 $, (ii) $\Leftrightarrow$ (iv) follows from
Remark~\ref{Obs definicion de mdoa}(iii).
\end{proof}
\begin{rem}\label{obsteor331}
If $\e < 1 $, from condition (iii),
$[\,\widehat{\phi},\widehat{\widetilde{\phi}}\,](\ga) \neq 0 $ for
a.e. $\ga\in N(\Phi)^{c}$. So $N(\widetilde{\Phi}) \subseteq
N(\Phi)$ a.e..
\end{rem}
Let $\{T_k\phi\}_{k\in\Z}$ be a Bessel sequence,
$\mathcal{W}=\overline{{\rm span}} \{T_k\phi\}_{k\in\Z}$ and $\phi_1
\in\Lebcuad$. The theorem below gives conditions so that there
exists $\widetilde{\phi}\in\overline{{\rm
span}}\{T_k\phi_1\}_{k\in\Z}$ whose translates allow approximate
reconstruction in $\mathcal{W}$. Furthermore, it yields a result
about approximate dual frames in shift-invariant subspaces.
\begin{thm}\label{caractdualestraslaH}
Let $\phi, \phi_1\in\Lebcuad$ be such that $\{T_k\phi\}_{k\in\Z}$
and $\{T_k\phi_{1}\}_{k\in\Z}$ are Bessel sequences. Let $\Wsb
=\overline{{\rm span}} \{T_k\phi\}_{k\in\Z}$ and $\Vsb
=\overline{{\rm span}}\{T_k\phi_1\}_{k\in\Z}$. Let
$\widetilde{\phi_1} \in\Wsb$ and $\widetilde{\phi}\in\Vsb $ such
that $\widehat{\widetilde{\phi}_1}
=\overline{\widetilde{H}}\widehat{\phi}$ and $\widehat{\widetilde{
\phi}} = \widetilde{H}\widehat{ \phi_1}$ where $\widetilde{H} \in
\mathcal{C}_{{\rm per}}^{\infty}$. Assume $\e\geq0 $. Then
$\{T_k\widetilde{\phi}\}_{k\in\Z}$ and
$\{T_k\widetilde{\phi}_1\}_{k\in\Z}$ are Bessel sequences and the
following are equivalent:
\begin{enumerate}
\item[(i)]
$\| f - \sum_{k\in\Z}\langle f, T_k\widetilde{\phi} \rangle T_k\phi
\| \leq \e \|f\| $ for $f\in\Wsb$.

\item[(ii)]
$\left|\overline{\widetilde{H}}(\gamma)
[\,\widehat{\phi},\widehat{\phi_{1}}\,](\ga) - 1\right| \leq \e$ for
a.e. $\ga\in N(\Phi)^{c}$.
\end{enumerate}
Moreover, if (ii) holds with  $\e < 1 $ and $N(\Phi_{1})=N(\Phi)$,
then $\{T_k\phi\}_{k\in\Z}$ and $\{T_k
P_{\Wsb}\widetilde{\phi}\}_{k\in\Z}$ are $\e$-approximate dual
frames in $\Wsb$, whereas $\{T_k\phi_1\}_{k\in\Z}$ and $\{T_k
P_{\Vsb } \widetilde{\phi}_1\}_{k\in\Z}$ are $\e$-approximate dual
frames in $\Vsb$.
\end{thm}
\begin{proof}
Observe that

\centerline{$\widetilde{\Phi}(\ga)=\sum_{k\in\Z}|\widetilde{H}(\gamma)
\widehat{\phi_1}(\gamma +
k)|^2=|\widetilde{H}(\gamma)|^{2}\Phi_1(\ga)\leq
||\widetilde{H}||_{L^\infty(0,1)}^2 \beta_1,$}

\noindent where $\beta_{1}$ is the Bessel bound of
$\{T_k\phi_{1}\}_{k\in\Z}$. So, by Theorem~\ref{T Benedetto}(i),
$\{T_k\widetilde{\phi}\}_{k\in\Z}$ is a Bessel sequence. Analogously
it can be seen that $\{T_k\widetilde{\phi}_1\}_{k\in\Z}$ is a Bessel
sequence.

Since $\abs{[\,\widehat{\phi},\widehat{\widetilde{\phi}}\,](\ga) - 1
}=\left|\overline{\widetilde{H}}(\gamma)
[\,\widehat{\phi},\widehat{\phi_1}\,](\ga) - 1\right|$, (i)
$\Leftrightarrow$ (ii) follows from
Theorem~\ref{caractdualestrasla}.

The proof of the last part is similar to the proof of (iii)
$\Rightarrow$ (i) of Theorem~\ref{caractdualestrasla}.
\end{proof}

With $\e = 0 $, from Theorem~\ref{caractdualestrasla} we obtain
\cite[Theorem 4.1.]{Christensen-Eldar (2004)} and from
Theorem~\ref{caractdualestraslaH} we obtain \cite[Theorem
4.3.]{Christensen-Eldar (2004)}.

%-------------------------------------------------------------------
\section{Approximate oblique dual frames in shift-invariant spaces}

In this section we study the concept of approximate oblique duality
when  $\mathcal{H}= \Lebcuad$ and $\Wsb$ and $\Vsb$ are
shift-invariant subspaces of $\Lebcuad$.

The following corollaries give conditions on the generators of two
shift-invariant subspaces for approximate oblique duality. The first
is a consequence of Theorem~\ref{caractdualestrasla} and
Proposition~\ref{P aprox subesp entonces aodf}.
\begin{cor}\label{C cond suf doa}
 Let $\Lebcuad =\Wsb \oplus\Vsb^{\perp}$. Let $\phi, \widetilde{\phi}\in\Lebcuad$ be such that
$\{T_k\phi\}_{k\in\Z}$ and $\{T_k\widetilde{\phi}\}_{k\in\Z}$ are
frames for $\Wsb$ and $\Vsb$ respectively. Assume $\e\geq0 $. If

\centerline{
$\abs{[\,\widehat{\phi},\widehat{\widetilde{\phi}}\,](\ga) - 1} \leq
\frac{\e}{\norm{\pi_{\Wsb\Vsb^{\perp}}}}$ for a.e. $\ga\in
N(\Phi)^{c}$}

\noindent then $\{T_k\phi\}_{k\in\Z}$ and
$\{T_k\widetilde{\phi}\}_{k\in\Z}$ are \emph{$\e $}-approximate
oblique dual frames.
\end{cor}
Let $\phi, \phi_1\in\Lebcuad$ be such that $\{T_k\phi\}_{k\in\Z}$
and $\{T_k\phi_1\}_{k\in\Z}$ are frames for $\Wsb$ and $\Vsb$
respectively. For what follows we recall that, by
Proposition~\ref{cond.L2.en.suma.dir}, if $\Lebcuad =\Wsb
\oplus\Vsb^{\perp}$ then $N(\Phi)=N(\Phi_{1})$.
\begin{cor}\label{C cond suf doaH}
 Let $\Lebcuad =\Wsb \oplus\Vsb^{\perp}$. Let $\phi,
\phi_1\in\Lebcuad$ be such that $\{T_k\phi\}_{k\in\Z}$ and
$\{T_k\phi_1\}_{k\in\Z}$ are frames for $\Wsb$ and $\Vsb$
respectively. Let $\widetilde{\phi}\in\Vsb $ be such that
$\widehat{\widetilde{ \phi}} = \widetilde{H}\widehat{\phi_1}$ where
$\widetilde{H} \in \mathcal{C}_{{\rm per}}^{2}$. Assume $0 \leq \e <
1$. If

\centerline{ $\left|\overline{\widetilde{H}}(\gamma)
[\,\widehat{\phi},\widehat{\phi_{1}}\,](\ga) - 1\right| \leq
\frac{\e}{\norm{\pi_{\Wsb\Vsb^{\perp}}}}$ for a.e. $\ga\in
N(\Phi)^{c}$ }

\noindent then $\{T_k\widetilde{\phi}\}_{k\in\Z}$ is a frame for
$\Vsb$ such that $\{T_k\phi\}_{k\in\Z}$ and
$\{T_k\widetilde{\phi}\}_{k\in\Z}$ are \emph{$\e $}-approximate
oblique dual frames.
\end{cor}
\begin{proof}
By Proposition~\ref{cond.L2.en.suma.dir}, there exists $c
>0$ such that

\centerline{ $\left|\widetilde{H}(\gamma)\right| \leq
\frac{1}{|[\,\widehat{\phi},\widehat{\phi_{1}}\,](\ga)|}\left(1+\frac{\e}{\norm{\pi_{\Wsb\Vsb^{\perp}}}}\right)
\leq
\frac{1}{c}\left(1+\frac{\e}{\norm{\pi_{\Wsb\Vsb^{\perp}}}}\right).$}
\noindent for a.e. $\ga\in N(\Phi)^{c}$. So, by Theorem~\ref{T
Benedetto}(i),

\centerline{$\widetilde{\Phi}(\ga)=|\widetilde{H}(\gamma)|^2\Phi_1(\ga)\leq
\frac{1}{c^2}\left(1+\frac{\e}{\norm{\pi_{\Wsb\Vsb^{\perp}}}}\right)^2
\beta_1,$}

\noindent where $\beta_{1}$ is an upper frame bound of $\{
T_k\phi_{1}\}_{k\in\Z}$. Applying again Theorem~\ref{T Benedetto}(i)
we conclude that $\{T_k\widetilde{\phi}\}_{k\in\Z}$ is a Bessel
sequence.

Now the conclusion follows from Theorem~\ref{caractdualestraslaH},
Proposition~\ref{P aprox subesp entonces aodf} and Remark~\ref{Obs
definicion de mdoa}(iii).
\end{proof}
The following theorem gives an expression for the Fourier transform
of the oblique projection when the subspaces are shift-invariant, in
terms of the corresponding generators.
\begin{thm}\label{caractproyoblic} Let $\Lebcuad =\Wsb \oplus\Vsb^{\perp}$. Let $\phi, \phi_1\in\Lebcuad$ be such that $\{T_k\phi\}_{k\in\Z}$ and
$\{T_k\phi_1\}_{k\in\Z}$ are frames for $\Wsb$ and $\Vsb$
respectively. Then
$$\widehat{\pi_{\Wsb\Vsb^{\perp}}f }(\ga) =\left\{
\begin{array}{cc}
\frac{[\,\widehat{f},\widehat{\phi_{1}}\,](\ga)}{
[\,\widehat{\phi},\widehat{\phi_{1}}\,](\ga)}
\widehat{\phi}(\ga) &\ga\in N(\Phi)^{c} \\
0 &\ga\in N(\Phi)
\end{array}
\right.
$$
for $f\in\Lebcuad$.
\end{thm}
\begin{proof}
Let $f\in\Lebcuad$. Since $\{T_k\phi\}_{k\in\Z}$ is a frame for
$\Wsb$ there exists $\tau \in \mathcal{C}_{{\rm per}}^{2}$ such that
$\widehat{\pi_{\Wsb\Vsb^{\perp}}f} = \tau \widehat{\phi}.$ So, if
$\ga\in  N(\Phi)$,
 $\widehat{ \pi_{\Wsb\Vsb^{\perp}}f }(\ga) =0.$ By \cite[Lemma 2.8]{de Boor-DeVore-Ron (1994)},
$$[\,\widehat{f},\widehat{\phi_{1}}\,](\ga)=[\,\widehat{\pi_{\Wsb\Vsb^{\perp}}f},\widehat{\phi_{1}}\,](\ga)=\tau(\ga)[\,\widehat{\phi},\widehat{\phi_{1}}\,](\ga).$$
If $\ga\in N(\Phi)^{c}$, by Proposition~\ref{cond.L2.en.suma.dir},
$[\,\widehat{\phi},\widehat{\phi_{1}}\,](\ga) \neq 0$, hence
$\tau(\ga) = \frac{[\,\widehat{f},\widehat{\phi_{1}}\,](\ga)}{
[\,\widehat{\phi},\widehat{\phi_{1}}\,](\ga)}$ and $\widehat{\pi_{
\Wsb\Vsb^{\perp}}f}(\ga) =
\frac{[\,\widehat{f},\widehat{\phi_{1}}\,](\ga)}{
[\,\widehat{\phi},\widehat{\phi_{1}}\,](\ga)}\widehat{\phi}(\ga)$.\qedhere
\end{proof}
If $\Wsb =\Vsb$ the previous theorem reduces to \cite[Theorem
2.9]{de Boor-DeVore-Ron (1994)}. Using the expression for the
Fourier transform of the oblique projection given in
Theorem~\ref{caractproyoblic}, we obtain the following sufficient
condition for approximate duality which is different to the one of
Corollary~\ref{C cond suf doaH}. If $\{T_k\phi\}_{k\in\Z}$ and
$\{T_k\phi_1\}_{k\in\Z}$ are Riesz bases for $\Wsb$ and $\Vsb$
respectively, we have an equality in (\ref{E cos marco})
\cite{Unser-Aldroubi (1994)}. So, in this case the sufficient
condition of the next theorem is weaker than the one of
Corollary~\ref{C cond suf doaH}.
\begin{thm}\label{condsufdoaparamtras}
 Let $\Lebcuad =\Wsb \oplus\Vsb^{\perp}$. Let $\phi,
\phi_1\in\Lebcuad$ be such that $\{T_k\phi\}_{k\in\Z}$ and
$\{T_k\phi_1\}_{k\in\Z}$ are frames for $\Wsb$ and $\Vsb$
respectively. Let $0 \leq \e < 1$. If $\widetilde{\phi}\in\Vsb $
verifies $\widehat{\widetilde{\phi}} =
\widetilde{H}\widehat{\phi_1},$ where $\widetilde{H} \in
\mathcal{C}_{{\rm per}}^{2}$ is such that
\begin{equation}\label{condpH}
\left|\overline{\widetilde{H}}(\ga)[\,\widehat{\phi},\widehat{\phi_{1}}\,](\ga)
- 1\right| \leq
\e\frac{\left|[\,\widehat{\phi},\widehat{\phi_{1}}\,](\ga)\right|}{
\sqrt{\Phi(\ga)\Phi_1(\ga)}}
\end{equation}
for a.e. $\ga\in N(\Phi)^{c}$, then $\{T_k\phi\}_{k\in\Z}$ and
$\{T_k\widetilde{\phi}\}_{k\in\Z}$ are $\e $-approximate oblique
dual frames. \qedhere
\end{thm}
\begin{proof}
Let $f\in\Lebcuad$. Using Theorem~\ref{caractproyoblic} and
Lemma~\ref{L saca f afuera}(i), we obtain
\begin{align*}
\|\pi_{\Wsb\Vsb^{\perp}}f - \sum_{k\in\Z}\langle f,
T_k\widetilde{\phi}\rangle
T_k\phi\|^2&=\|\widehat{\pi_{\Wsb\Vsb^{\perp}}f} -
\sum_{k\in\Z}\langle
\widehat{f},\widehat{T_{k}\widetilde{\phi}}\rangle\widehat{T_{k}\phi} \|^2\\
&=
\int_{N(\Phi)^{c}}\left|\frac{[\,\widehat{f},\widehat{\phi_{1}}\,](\ga)}{[\,\widehat{\phi},\widehat{\phi_{1}}\,](\ga)}
\widehat{\phi}(\ga)-\widehat{\phi}(\ga)[\,\widehat{f},\widehat{\widetilde{
\phi}}\,](\ga)
\right|^2 \mathrm{d}\ga\\
&= \int_{N(\Phi)^{c}}\left|
[\,\widehat{f},\widehat{\phi_{1}}\,](\ga)\widehat{\phi}(\ga)\left(\frac{1}{[\,\widehat{\phi},\widehat{\phi_{1}}\,](\ga)}
-\overline{\widetilde{H}}(\ga)\right)
\right|^2\mathrm{d}\ga\\
&= \int_{N(\Phi)^{c} \cap [0,1)}\Phi(\ga)\left|
[\,\widehat{f},\widehat{\phi_{1}}\,](\ga)\left(\frac{1}{[\,\widehat{\phi},\widehat{\phi_{1}}\,](\ga)}
-\overline{\widetilde{H}}(\ga)\right)
\right|^2\mathrm{d}\ga\\
&\leq \int_{N(\Phi)^{c} \cap [0,1)}
\Phi(\ga)\Phi_{1}(\ga)\sum_{n\in\Z} | \widehat{f}(\ga+n)
|^2\left|\frac{1}{[\,\widehat{\phi},\widehat{\phi_{1}}\,](\ga)}
-\overline{\widetilde{H}}(\ga)
\right|^2\mathrm{d}\ga\\
&\leq \e^2 \int_{0}^{1}\sum_{n\in\Z}|\widehat{f}(\ga+n)|^2
\mathrm{d}\ga = \e^2 \|f\|^2.
\end{align*}
Hence, by Definition~\ref{D marcos duales oblicuos aproximados} and
Remark~\ref{Obs definicion de mdoa}(iii), $\{T_k\phi\}_{k\in\Z}$ and
$\{T_k\widetilde{\phi}\}_{k\in\Z}$ are $\e $-approximate oblique
dual frames.
\end{proof}
\begin{rem}\label{obscondsufdoaparamtras}
Assume that (\ref{condpH}) does not hold. Then there exists
$E\subseteq N(\Phi)^{c} \cap [0,1)$ with $|E|> 0$ and that

\centerline{$\left|\overline{\widetilde{H}}(\ga)[\,\widehat{\phi},\widehat{\phi_{1}}\,](\ga)
- 1\right|>
\e\frac{\left|[\,\widehat{\phi},\widehat{\phi_{1}}\,](\ga)\right|}{\sqrt{\Phi(\ga)\Phi_1(\ga)}}$}

\noindent for $\ga\in E $. Let $\e'>0 $ with $\e' <
\frac{c\e}{\sqrt{ \beta \beta_1}}$ (where $c>0 $ is the constant of
Proposition~\ref{cond.L2.en.suma.dir} and $\beta, \beta_1$ are upper
frame bounds of $\{T_k\phi\}_{k\in\Z}$ and $\{T_k\phi_1\}_{k\in\Z}$,
respectively). Note that $\e' <
\frac{c\e}{\sqrt{\Phi(\ga)\Phi_1(\ga)}} \leq
\frac{c\e}{|[\,\widehat{\phi},\widehat{\phi_{1}}\,](\ga)|} \leq \e
$. Let $E_{p}=\bigcup_{k\in \mathbb{Z}}(E+k)$ and $F \in
\mathcal{C}_{{\rm per}}^{2}$ such that $\textrm{supp}(F)\subseteq
E_{p}$ and $\abs{\textrm{supp}(F)}
>0$. Let $f\in\Lebcuad$ given by $\widehat{f}=F\widehat{\phi}$.
Then $f\in\Wsb$, $f\neq0$ and $\pi_{\Wsb\Vsb^{\perp}}f=f$. Taking
into account Lemma~\ref{L saca f afuera}(ii), we get
\begin{align*}
\|\pi_{\Wsb\Vsb^{\perp}}f - \sum_{k\in\Z}\langle f,
T_k\widetilde{\phi} \rangle T_k\phi \|^2 &=
\int_{N(\Phi)^{c}}\left|\widehat{f}(\ga)\right|^{2}\left|1
-\overline{\widetilde{H}}(\ga)
[\,\widehat{\phi},\widehat{\phi_1}\,](\ga)\right|^2 \mathrm{d}\ga\\
&\geq \e^2 \int_{N(\Phi)^{c}}|\widehat{f}(\ga) |^2
\frac{\left|[\,\widehat{\phi},\widehat{\phi_1}\,](\ga)\right|^2}{\Phi(\ga)
\Phi_1(\ga)}\mathrm{d}\ga \\
&\geq \paren{\frac{c\e}{\sqrt{ \beta \beta_1}}}^{2}
\int_{N(\Phi)^{c}}|\widehat{f}(\ga) |^2 \mathrm{d}\ga > {\e'}^{\,2}
||f||^2.
\end{align*}
This shows that $\{T_k\phi\}_{k\in\Z}$ and
$\{T_k\widetilde{\phi}\}_{k\in\Z}$ are not $\e' $-approximate
oblique dual frames.
\end{rem}
We have the following necessary condition for approximate oblique
duality.
\begin{thm}\label{condnecdoaparamtras}
 Let $\Lebcuad =\Wsb \oplus\Vsb^{\perp}$. Let $\phi,
\phi_1\in\Lebcuad$ be such that $\{T_k\phi\}_{k\in\Z}$ and
$\{T_k\phi_1\}_{k\in\Z}$ are frames for $\Wsb$ and $\Vsb$
respectively. Set $\widetilde{\phi}\in\Vsb $ such that
$\widehat{\widetilde{ \phi}} = \widetilde{H} \widehat{\phi_1}$,
where $\widetilde{H} \in \mathcal{C}_{{\rm per}}^{2}$. Let
$\e\geq0$. If $\{T_k\phi\}_{k\in\Z}$ and
$\{T_k\widetilde{\phi}\}_{k\in\Z}$ are $\e$-approximate oblique dual
frames, then
\begin{equation}\label{condnecpH}
\left|\overline{\widetilde{H}}(\ga)[\,\widehat{\phi},\widehat{\phi_1}\,](\ga)
- 1 \right| \leq \e
\end{equation}
for a.e. $\ga\in N(\Phi)^{c}$. \qedhere
\end{thm}
\begin{proof}
We proceed in a similar way to the proof of
Theorem~\ref{caractdualestrasla}. Assume (\ref{condnecpH}) does not
hold. Then there exists $\e'
>\e$ and $E'\subseteq N(\Phi)^{c} \cap [0,1)$ such that $|E'|>0 $ and

\centerline{$\left|\overline{\widetilde{H}}(\ga)[\,\widehat{\phi},\widehat{\phi_1}\,](\ga)-1\right|\geq\e'$
for $\ga\in E'$.}

Let $E'_{p}=\bigcup_{k\in \mathbb{Z}}(E'+k)$ and $F \in
\mathcal{C}_{{\rm per}}^{2}$ such that $\textrm{supp}(F) \subseteq
E'_{p}$ and $\abs{\textrm{supp}(F)} >0$. Let $f\in \Wsb$ with
$\widehat{f} = F\widehat{\phi}$. In this case,
$\pi_{\Wsb\Vsb^{\perp}}f = f$ and $f\neq0$. Using Lemma~\ref{L saca
f afuera}(ii), we obtain
\begin{align*}
\|\pi_{\Wsb\Vsb^{\perp}}f - \sum_{k\in\Z}\langle f,
T_k\widetilde{\phi}\rangle T_k\phi\|^2 =&
\int_{N(\Phi)^{c}}\left|\widehat{f}(\ga)\right|^{2}\left|1
-\overline{\widetilde{H}}(\ga)
[\,\widehat{\phi},\widehat{\phi_1}\,](\ga)\right|^2 \mathrm{d}\ga\\
\geq& {\e'}^{\,2} \int_{N(\Phi)^{c}}|\widehat{f}(\ga)|^2
\mathrm{d}\ga = {\e'}^{\,2} \int|\widehat{f}(\ga)|^2 \mathrm{d}\ga
>\e^2 ||f||^2.
\end{align*}
This contradicts the fact that $\{T_k\phi\}_{k\in\Z}$ and
$\{T_k\widetilde{\phi}\}_{k\in\Z}$ are $\e$-approximate oblique dual
frames. \qedhere
\end{proof}

%%%%%%%%%%%%%%%%%%%%%%%%%%%%%%%%%%%%%%%%%%%%%%%%%%%%%%%%%%%%%%%%%%%
\section{Numerical and computational aspects of approximate oblique
dual frames}

In this section we highlight the importance of approximate oblique
dual frames from a numerical and computational point of view,
illustrating our analysis with an example of $B$-splines.

Assume that $\Lebcuad=\Wsb\oplus\Vsb^{\perp}$. Let $\phi,
\phi_1\in\Lebcuad$ be such that $\{T_k\phi\}_{k\in\Z}$ and
$\{T_k\phi_1\}_{k\in\Z}$ are frames for $\Wsb$ and $\Vsb$
respectively. By \cite[Theorem 4.3]{Christensen-Eldar (2004)} and
Proposition~\ref{cond.L2.en.suma.dir}, the unique oblique dual frame
of $\{T_k\phi\}_{k\in\Z}$ in $\Vsb$ is $\{T_k\psi\}_{k\in\Z}$, where
$\psi=\sum_{k\in\Z}c_k T_k\phi_1$ and $\{c_k\}_{k\in\Z}$ are the
Fourier coefficients of

\centerline{$H(\ga)=\frac{1}{[\,\widehat{\phi},\widehat{\phi_1}\,](\ga)}$.}

\noindent In practice, the generator of the oblique dual frame can
only be obtained approximately since the computation of the previous
series and the calculus of the Fourier coefficients are not exact.
So, in fact, in the applications we always have to work with an
approximate oblique dual frame. In the following example we
illustrate this situation. We consider here the effect of the
truncation required for the series that defines $\psi$ when $H$ is
not a trigonometric polynomial. In this case, the used
$\e(K)$-approximate dual frame is $\{T_k\psi_{K}\}_{k\in\Z}$ where
$\psi_{K}=\sum_{|k|\leq K}c_k T_k\phi_1$ for some $K\in \N$. Taking
into account (\ref{condpH}) we have the error

\centerline{$\e(K) = \sup_{\ga\in [0, 1]}\sqrt{\Phi(\ga)
\Phi_1(\ga)}\left|\overline{H}_{K}(\ga) - H(\ga)\right|$}

\noindent where $H_{K}\in \mathcal{C}_{{\rm per}}^{\infty}$ and
restricted to $[0,1)$ is equal to $\sum_{|k|\leq K }c_k e^{-2i\pi
k\cdot}$.

Consider now in particular shift-invariant subspaces generated by
$B$-splines (see \cite[Example 4.5]{Christensen-Eldar (2004)}). We
recall that $B$-splines $B_{n}$ are functions which are piecewise
polynomials. They are defined inductively as

\centerline{$B_{1}(x)=\chi_{\corch{-\frac{1}{2},-\frac{1}{2}}}(x),\,\,B_{n+1}(x)=\int_{-\infty}^{\infty}B_{n}(x-t)B_{1}(t)dt=\int_{-\frac{1}{2}}^{\frac{1}{2}}B_{n}(x-t)dt.$}

\noindent For each $n\geq2$,

\centerline{$B_{n}(x)=\frac{n!}{\paren{n-1}!}\sum_{j=0}^{n}\paren{-1}^{j}\frac{1}{j!\paren{n-j}!}
\paren{x+\frac{n}{2}-j}_{+}^{n-1}$, $x\in \mathbb{R}$,}

\noindent where $f(x)_{+}=\max\set{f(x),0}$ if $f: \mathbb{R}
\rightarrow \mathbb{R}$. $B$-splines have the following properties:
$B_{n}\in C^{n-2}\paren{\mathbb{R}}$ for $n\geq2$, ${\rm
supp}\paren{B_{n}}=\corch{-\frac{n}{2},-\frac{n}{2}}$ and $B_{n}>0$
on $\paren{-\frac{n}{2},-\frac{n}{2}}$. See e.g.
\cite[A.8]{Christensen (2016)} for more details.

Let $n, m\in \Z$, $\phi=B_n$ and $\phi_1=B_{n+2m}$. Here,

\centerline{$H(\ga)=\frac{1}{\sum_{l\in\Z}
\left(\frac{\sin{\pi(\ga+l)}}{\pi(\ga+l)}\right)^{2(n+m)}}$}

\noindent is not a trigonometric polynomial (see \cite[Example
9.4.3]{Christensen (2016)}). So, the generator of the oblique dual
frame of $\{T_k B_n\}_{k\in\Z}$ in $\Vsb$ has not compact support
whereas for the generator $\psi_{K}$ of the approximate oblique dual
frame we have $\textrm{supp}(\psi_{K})= [-\frac{n+2m+2K}{2},
\frac{n+2m+2K}{2}]$. We also note that $\psi_{K}$ has the same
smoothness as $B_{n+2m}$.

In the figures, we show the plots obtained with MATLAB of $H$ (solid
line), $H_{K}$ (dash-dot line), $\psi$ (solid line) and $\psi_{K}$
(dash-dot line) for $K = 0, 1, 2, 3$. Figure~\ref{figure:13}
corresponds to the case $n=1$ and $m=3$ and Figure~\ref{figure:21}
to the case $n=2$ and $m=1$.

\begin{figure}[!]
\centering
\includegraphics[width=1\textwidth]{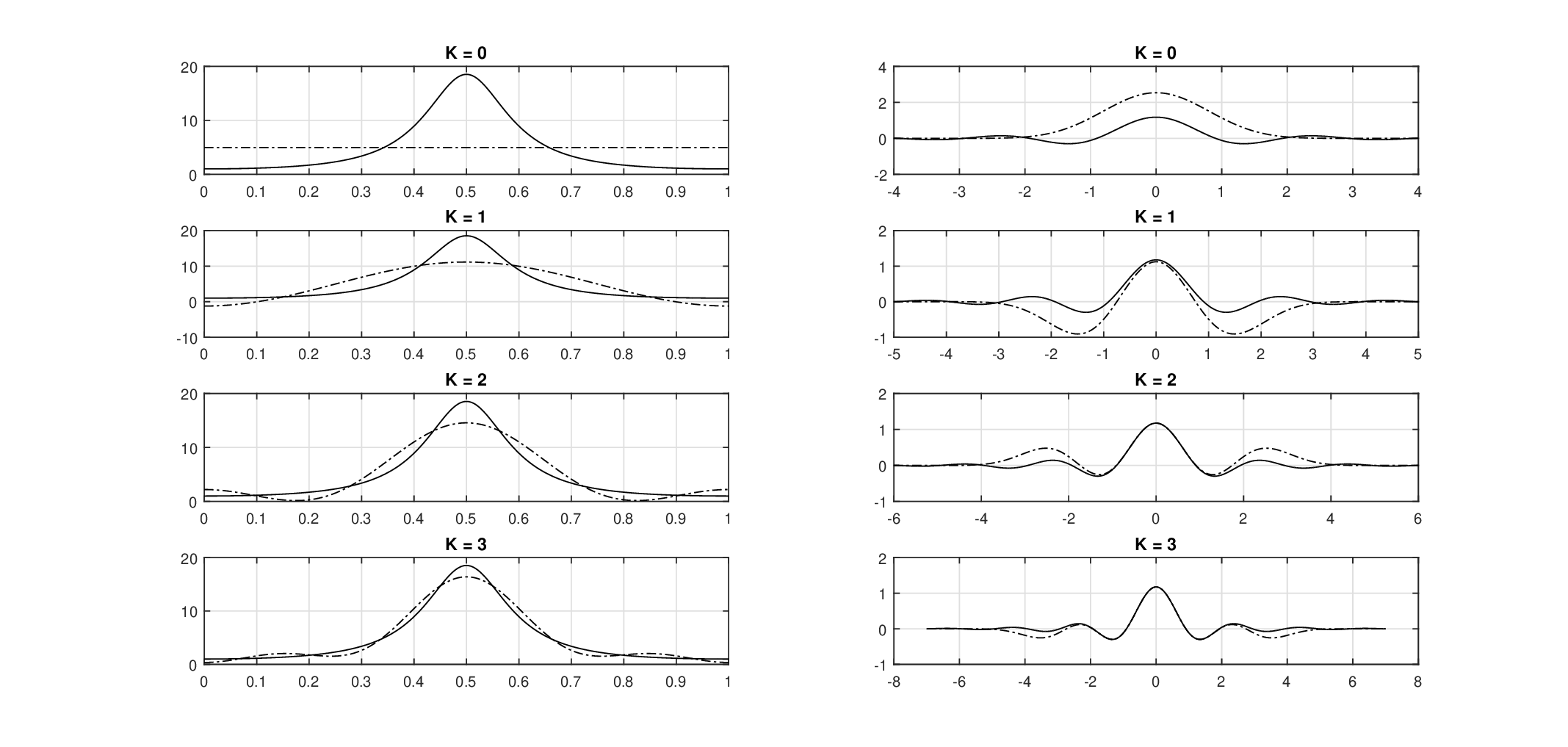}
\caption{{\scriptsize Case $n=1$ and $m=3$. (a) The functions $H$
(solid line) and $H_{K}$ (dash-dot line); (b) The generators $\psi$
(solid line) and $\psi_{K}$ (dash-dot line).}}\label{figure:13}
\end{figure}

\begin{figure}[!]
\centering
\includegraphics[width=1\textwidth]{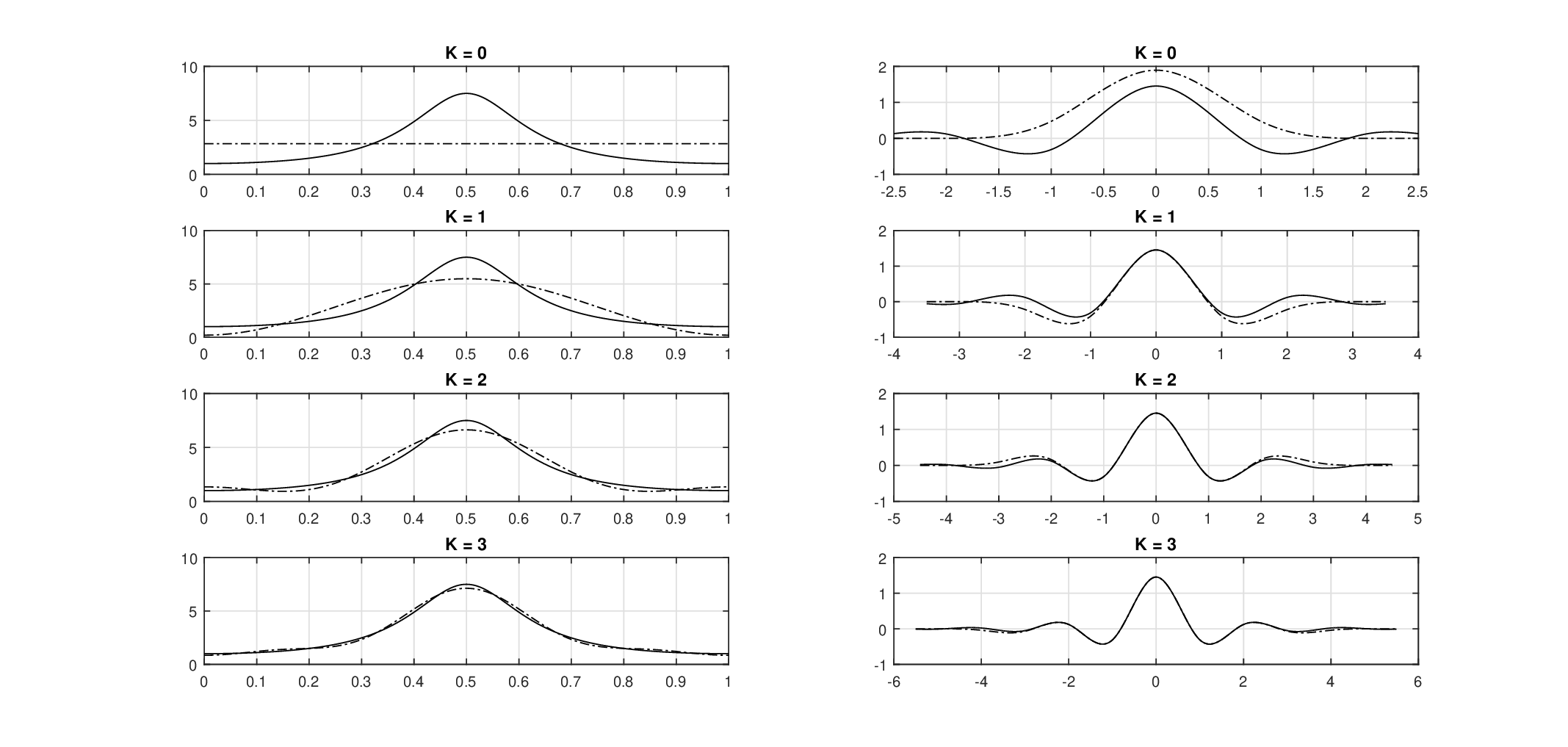}
\caption{{\scriptsize Case $n=2$ and $m=1$. (a) The functions $H$
(solid line) and $H_{K}$ (dash-dot line); (b) The generators $\psi$
(solid line) and $\psi_{K}$ (dash-dot line).}}\label{figure:21}
\end{figure}

According to the expression of $H$ and $H_{K}$, the errors $\e(K)$
depend on $n+m$. The entries of the tables that appear below were
computed using MATLAB. The first three show the supports of
$\psi_{K}$, the absolute values $\abs{c_{K+1}}$, the norms
$\norm{\psi_{K}-\psi}$ and the estimated errors $\varepsilon(K)$ for
$K = 0, 1, \ldots, 6$. Table~\ref{table:13} corresponds to the case
$n=1$ and $m=3$ and Table~\ref{table:21} to the case $n=2$ and
$m=1$. We observe that $\abs{c_{K+1}} \rightarrow 0$ as $K
\rightarrow 0$. This convergence is slower when $n+m$ increases.
This is reflected in the behavior of $\norm{\psi_{K}-\psi}$ and
$\varepsilon(K)$ which tend to zero more slowly when $n+m$ increases
too. The error $\varepsilon(K)$ decreases exponentially. Applying
the function polyfit to $\log_{2}\varepsilon(K)$ we obtained
approximate expressions of $\varepsilon(K)$ for $n+m= 1, \ldots, 8$
that appear in Table~\ref{table:e}.

%\twocolumn
%\begin{multicols}{2}
%[]
{\scriptsize
\begin{table}[!]
\centering
\begin{tabular}{|c |c |c | c | c |}
 \hline
 $K$ & $\textrm{supp}(\psi_{K})$ & $\abs{c_{K+1}}$ & $\norm{\psi_{K}-\psi}$ & $\e(K)$ \\ [1ex]
 \hline
 0 & $[-3.5, 3.5]$ & $3.0910$ & $4.3233$ & $3.9647$ \\[1ex]
 \hline
 1 & $[-4.5, 4.5]$ & $1.7080$ & $2.4003$ & $2.2172$ \\[1ex]
 \hline
 2 & $[-5.5, 5.5]$ & $0.9208$ & $1.2954$ & $1.1984$  \\[1ex]
 \hline
 3 & $[-6.5, 6.5]$ & $0.4937$ & $0.6947$ & $0.6428$  \\[1ex]
 \hline
 4 & $[-7.5, 7.5]$ & $0.2644$ & $0.3720$ & $0.3442$  \\[1ex]
 \hline
 5 & $[-8.5, 8.5]$ & $0.1415$ & $0.1992$ & $0.1842$  \\[1ex]
 \hline
 6 & $[-9.5, 9.5]$ & $0.0758$ & $0.1066$ & $0.0986$  \\[1ex]
 \hline
\end{tabular}
\smallskip
\caption{{\scriptsize Case $n=1$ and $m=3$.}}\label{table:13}
\end{table}
}

%\columnbreak

{\scriptsize
\begin{table}[!]
\centering
\begin{tabular}{|c |c |c | c | c |}
 \hline
 $K$ & $\textrm{supp}(\psi_{K})$ & $\abs{c_{K+1}}$ & $\norm{\psi_{K}-\psi}$ & $\e(K)$ \\ [1ex]
 \hline
 0 & $[-2, 2]$ & $1.3217$ & $2.2396$ & $1.8421$ \\[1ex]
 \hline
 1 & $[-3, 3]$ & $0.5733$ & $0.9715$ & $0.8012$ \\[1ex]
 \hline
 2 & $[-4, 4]$ & $0.2470$ & $0.4186$ & $0.3453$  \\[1ex]
 \hline
 3 & $[-5, 5]$ & $0.1064$ & $0.1803$ & $0.1487$  \\[1ex]
 \hline
 4 & $[-6, 6]$ & $0.0458$ & $0.0776$ & $0.0640$  \\[1ex]
 \hline
 5 & $[-7, 7]$ & $0.0197$ & $0.0334$ & $0.0276$  \\[1ex]
 \hline
 6 & $[-8, 8]$ & $0.0085$ & $0.0144$ & $0.0119$  \\[1ex]
 \hline
\end{tabular}
\smallskip
\caption{{\scriptsize Case $n=2$ and $m=1$.}}\label{table:21}
\end{table}
}

%\end{multicols}

%\onecolumn

{\scriptsize
\begin{table}[!]
\centering
\begin{adjustbox}{max width=\textwidth}
\begin{tabular}{|c |c |c |c |c |c |c |c | }
 \hline
 $n+m$ & $2$ & $3$& $4$& $5$& $6$& $7$& $8$ \\ [1ex]
 \hline
 $\e(K)$ & $2^{-1.9K}0.73$ & $2^{-1.21K}1.84$ & $2^{-0.89K}3.96$ & $2^{-0.69K}8.24$  & $2^{-0.56K}17.1$ & $2^{-0.46K}35.93$ & $2^{-0.38K}76.63$ \\[1ex]
 \hline
\end{tabular}
\end{adjustbox}
\smallskip
\caption{{\scriptsize Approximated expressions of $\e(K)$ for
different values of $n+m$.}}\label{table:e}
\end{table}
}

%%%%%%%%%%%%%%%%%%%%%%%%%%%%%%%%%%%%%%%%
\section*{Declaration of interest: none}

%%%%%%%%%%%%%%%%%%%%%%%%%%
\section*{Acknowledgement}
We thank the reviewers for the detailed reading of the paper and for
the very useful and valuable observations. This research has been
supported by Grants PIP 112-201501-00589-CO (CONICET), PROIPRO
03-1620 (UNSL), PICT-2014-1480 and UBACyT 20020130100422BA.

\end{document}